\documentclass[a4paper,11pt,reqno]{amsart}

\usepackage{hyperref}
\usepackage{amssymb}
\usepackage{eucal}
\usepackage{eufrak}
\usepackage{a4wide}
\usepackage{graphics}
\usepackage{graphicx}
\usepackage{epsfig}
\usepackage{multicol}
\usepackage{xypic}
\usepackage{amsmath}
\usepackage{wasysym}
\usepackage{dsfont}
\usepackage{float}
\usepackage{fancyhdr}
\usepackage{subfigure}

\usepackage{eufrak}
\usepackage{graphics}
\usepackage{graphicx}
\usepackage{epsfig}
\usepackage{multicol}
\usepackage{xypic}
\usepackage{amsmath}
\usepackage{color}

\newtheorem{proposition}{Proposition}[section]
\newtheorem{theorem}{Theorem}[section]

\theoremstyle{definition}
\newtheorem{definition}{Definition}[section]
\newtheorem{remark}{Remark}[section]

\newtheorem{corollary}[definition]{Corollary}

\newcommand{\J}{\mathds{J}}

\newcommand{\R}{\mathds{R}}      
\newcommand{\C}{\mathds{C}}      
\newcommand{\F}{\mathds{F}}
\newcommand{\I}{\mathds{I}}

\newcommand{\too}{\rightarrow}

\usepackage{amssymb}

\newcommand{\proa}{A^*G \mbox{$\;$}_{\tau^*} \kern-3pt\times_\alpha
G \mbox{$\;$}_\beta \kern-3pt\times_{\tau^*} A^*G}

\newcommand{\Ad}{\mbox{Ad}}

\newcommand{\ad}{\mbox{ad}}

\newcommand{\so}{\mathfrak{so}}

\newcommand{\al}{\mathfrak{g}}

\newcommand{\de}{\mathfrak{d}}
\newcommand{\Ese}{\mathcal{S}}

\newcommand{\De}{\mathcal{D}}

\begin{document}

\title[A case study of consistency for nonholonomic integrators]{The geometric discretisation of the Suslov problem: a case study of consistency for nonholonomic integrators}

\author[L.C. Garc\'ia-Naranjo]{Luis C. Garc\'ia-Naranjo}
\address{L.C. Garc\'ia-Naranjo: Departamento de Matem\'aticas y Mec\'anica, IIMAS-UNAM, Apdo. Postal: 20-726, Mexico City, 01000, Mexico.} \email{luis@mym.iimas.unam.mx}

\author[F. Jim\'enez]{Fernando Jim\'enez}
\address{F. Jim\'enez: Department of Applied Mathematics, University of Waterloo, 200 Univ. Avenue West, N2L 3G1,  Waterloo, Canada.} \email{fernando.jimenez.alburquerque@gmail.com}


\maketitle

\begin{abstract}
Geometric integrators for nonholonomic systems were introduced by Cort\'es and Mart\'inez in \cite{CoMa} by
proposing a discrete Lagrange-D'Alembert principle. Their approach is based on the
 definition of a discrete Lagrangian $L_d$ and a discrete
constraint space $D_d$. There is no recipe to construct these objects and the performance of the integrator is sensitive to their choice.

Cort\'es and Mart\'inez  \cite{CoMa} claim that choosing $L_d$ and  $D_d$ in a consistent manner with respect to a finite difference map is necessary to
guarantee an approximation of the continuous flow within a desired order of accuracy. 
Although, this statement is given without proof, similar versions of it have appeared recently in the literature.

We evaluate the importance of the consistency condition by comparing the performance of two different geometric integrators for the nonholonomic Suslov problem,
only one of which corresponds to a consistent choice of $L_d$ and  $D_d$. We prove that both integrators produce approximations of the same order, 
and, moreover, that the non-consistent discretisation outperforms the other  in numerical  experiments and in terms of energy preservation. Our results indicate that 
 the consistency of a discretisation 
might  not be the most relevant feature to consider  in the construction
of nonholonomic geometric integrators.

\end{abstract}

\section{Introduction}

An extension of the theory of variational integrators \cite{Vese,Mose,MarsdenWest} to nonholonomic systems was proposed  by Cort\'es and Mart\'inez \cite{CoMa}
by introducing a discrete version of the Lagrange-D'Alembert principle (DLA). Their approach requires  the definition of two objects. The first of them
is the  {\em discrete Lagrangian} $L_d$ which is a real valued function on the Cartesian product $Q\times Q$, where $Q$ is the configuration space. 
The second object involved in their discretisation
is the {\em discrete constraint space} $D_d$ which is a submanifold of  $Q\times Q$ that
is the discrete counterpart of the non-integrable distribution $\mathcal{D}\subset TQ$ defined by the nonholonomic constraints. 

The discrete dynamics defined by the  DLA algorithm are sensitive to the choice of the discrete Lagrangian $L_d$ and the
discrete constraint space $D_d$. While there is no recipe to construct these objects, Cort\'es and Mart\'inez suggest that they should be constructed in a {\em consistent}
manner with respect to a {\em finite difference map} (rigorous definitions for these concepts are given in \S\ref{SS:DiffMaps}).
 More precisely, in Remark
3.1 of their article they state without proof: ``To guarantee that the DLA algorithm approximates the
continuous flow within a desired order of accuracy, one should select the discrete Lagrangian $L_d$ and the discrete constraint space $D_d$ in a consistent way''.
Similar statements appear in other sources like  \cite{FeZen}, \cite{FeChap}. The use of consistent discretisations is also suggested in \cite{McPer},\cite{JiSch}, \cite{JiSch2}.

The  goal of  this paper is to examine the importance of the consistency condition by considering in detail the performance of two different discretisations of the 
classical nonholonomic Suslov problem, only one of which is consistent. Recall that the Suslov problem \cite{Suslov} is a simple example of a nonholonomic system that can be realised
physically  and
that exhibits some of the main features that distinguish nonholonomic from Hamiltonian systems, like the non-existence of a smooth invariant measure, and the presence of 
attracting and repelling periodic orbits on the  energy level sets of the system. 

Our analysis of the performance of the two integrators focuses on the calculation of their local truncation errors, on a discussion of their energy-preservation properties, and on the
execution of  numerical experiments. 

Our results, explained in more detail below, indicate that the consistent discretisation does not perform better than the other 
in any of the aspects described above.  Our research suggests that consistency 
is not the most relevant property to consider in order to construct geometric integrators for nonholonomic systems with an enhanced behaviour.

\subsection{Two different geometric discretisations of the Suslov problem}

 The 
configuration space for the Suslov problem is the Lie group $Q=G=SO(3)$, and both the Lagrangian and the constraints are invariant with respect
to left multiplication on $G$, making the problem into a classical example of an {\em LL-system}. The reduced dynamics  takes place  on the reduced velocity phase space, which is the two-dimensional subspace $\de$ of the Lie algebra $\mathfrak{g}=\so(3)$ that defines the constraint
distribution at the group identity, and is     governed by 
 the {\em Euler-Poincar\'e-Suslov} equations.

Both of the discretisations that we consider fall into the scheme proposed by Fedorov and Zenkov \cite{FeZen} in which the discrete objects $L_d$ and $D_d$
are chosen to be  invariant with respect to the diagonal left multiplication of $G$ on $G\times G$. As a consequence, the discrete constraint
space $D_d\subset G\times G$ is determined by a {\em discrete displacement subvariety} $\Ese \subset G$ which is the 
discrete version of the reduced velocity phase  space $\de \subset \mathfrak{g}$. Similarly, the discrete Lagrangian $L_d$ is determined
by a reduced discrete Lagrangian $\ell_d:G\to \R$. The discrete dynamics on $D_d$ drop to $\Ese$
and define  {\em discrete Euler-Poincar\'e-Suslov}
equations.

In our analysis we define the discrete displacement subvariety $\Ese$ as the image of $\de$ under the 
{\em Cayley transform}. This definition of  $\Ese$ is reminiscent of the
 work in \cite{FeZen}, where $\Ese$ is  chosen as the image of $\de$ under the exponential map. In fact, as we explain in  
\S\ref{S:l1}, both approaches  to define $\Ese$ are  equivalent since the images of $\de$ under the Cayley and the exponential map
coincide  on an open dense subset of  $SO(3)$ that contains the identity.
 
The advantage to consider the Cayley transform over the exponential map is that the inverse transformation can be
explicitly written down as a rational map from $SO(3)$ to $\so(3)$. 
Moreover, in our work, we interpret this inverse map  as a (reduced) difference
map $\psi$ which we use to define a reduced discrete Lagrangian that we denote by $\ell^{(\infty,\varepsilon)}_d$ which can be explicitly computed (here and in what follows
$\varepsilon>0$ is the time step). The discretisation of the Suslov problem resulting from the choice of $\Ese$ and $\ell^{(\infty,\varepsilon)}_d$ is consistent with respect to $\psi$.

An alternative choice of the (reduced) discrete Lagrangian 
is the one considered by Moser and Veselov  in their celebrated discretisation of the free rigid body \cite{Mose}.
We denote this discrete Lagrangian by $\ell^{(1,\varepsilon)}_d$. The discretisation of the Suslov problem resulting from the choice of $\Ese$ and $\ell^{(1,\varepsilon)}_d$  is a reparametrisation of the one considered by Fedorov and Zenkov in \cite{FeZen} and, as we show in
Proposition \ref{P:NonConsistent}, it  is 
{\em not} consistent.

\subsection{Local truncation error of the approximation of the continuous flow} 
In order to compare the discrete and the continuous (reduced) flows on a common space, 
it is necessary to pass to the momentum phase space $\so(3)^*$. At the continuous level, this passage is defined by the inertia tensor
of the body  $\I$ which is interpreted as a linear isomorphism between $\so(3)$ and $\so(3)^*$. The image $\de^*:=\I(\de)$ is a two-dimensional
subspace of $\so(3)^*$ that is the reduced momentum phase space for the continuous Euler-Poincar\'e-Suslov equations.

At the discrete level  one defines the {\em discrete Legendre transformation} 
$\F\ell_d:G\to \so(3)^*$ as the right trivialisation of the derivative
of  $\ell_d$.
The discrete Euler-Poincar\'e-Suslov dynamics in momentum variables takes place in {\em the momentum locus} $\mathfrak{u}:=\F\ell_d(\Ese)$ which
is a nonlinear, possibly non-smooth,
subvariety of $\so(3)^*$.
Since the definition of the momentum locus involves the discrete reduced Lagrangian, the momentum loci $\mathfrak{u}_\varepsilon^{(1)}$
and $\mathfrak{u}_\varepsilon^{(\infty)}$, defined respectively by $\ell^{(1,\varepsilon)}_d$ and $\ell^{(\infty,\varepsilon)}_d$, are different.

The two discretisations of the Suslov problem introduced 
above, give rise to discrete evolution maps $\mathcal{B}^*_{(1,\varepsilon)}$ and $\mathcal{B}^*_{(\infty,\varepsilon)}$ defined
respectively on $\mathfrak{u}_\varepsilon^{(1)}$ and $\mathfrak{u}_\varepsilon^{(\infty)}$  that approximate the continuous 
Euler-Poincar\'e-Suslov equations on $\de^*$, provided that the  time step $\varepsilon$ is sufficiently 
small. We  obtain asymptotic expansions for $\mathcal{B}^*_{(1,\varepsilon)}$ and $\mathcal{B}^*_{(\infty,\varepsilon)}$ as
$\varepsilon\to 0$, that allow us to conclude that the local truncation error in both cases is of second order
(Theorems \ref{T:AccuracyMosVes} and \ref{T:AccuracyCayley}).

\subsection{Energy preservation} As reported in \cite{FeZen}, the non-consistent discretisation, whose 
reduced constrained Lagrangian is $\ell^{(1,\varepsilon)}_d$, defines 
a multi-valued  map that 
exactly preserves the energy
of the system. 

On the other hand, the consistent discretisation corresponding to $\ell^{(\infty,\varepsilon)}_d$ defines a multi-valued map that only preserves the energy of the system if a certain, non-generic, condition
on the inertia tensor of the body holds. 

\subsection{Numerical experiments}
In  \S\ref{Simulations} we present a series of work-precision diagrams that illustrate the convergence of both discrete systems to the continuous one
as the time step $\varepsilon \to 0$. The only advantage that results by working with the consistent discretisation 
is that
the resulting algorithm is well-defined for larger values of $\varepsilon$. However, for a sufficiently small time step, the non-consistent discretisation appears to give
a better approximation of the continuous flow.

\subsection{Organisation of the paper}

The paper is structured as follows: in \S\ref{Preli}  we quickly recall the necessary ingredients for our developments and we give the definitions of
finite difference maps and of consistency of a discretisation. We also consider equivariant finite difference maps for 
systems on Lie groups and we give a characterisation of consistency in this special case. 
 In  \S\ref{SusPro} we review the main aspects of the Suslov problem and we give a Taylor expansion of its solutions that is later used to compute the local truncation
 error of the discretisations.
 In \S\ref{S:Pade} we define the two different discretisations of the Suslov problem and show that only one of them is consistent. In \S\ref{S:req1}  (respecively, \S\ref{S:reqInfty}) we give working
 formulae for the discrete algorithm of the non-compatible (respectively  compatible) discretisation and show that the local truncation error in the approximation of the continuous flow is second order.
In  \S\ref{S:Energy} we discuss the preservation of energy for both discretisations and in  \S\ref{Simulations} we present some numerical results. The main conclusions of the paper  are 
given in \S\ref{S:Conclusions}. Finally, we present an appendix that shows that the momentum loci defined by both discretisations are contained in the zero level sets of  certain polynomials.

\section{Preliminaries}\label{Preli}

\subsection{Nonholonomic systems}
\label{S:LA}

A nonholonomic  system  on an $n$ dimensional configuration manifold $Q$ is determined by the triple $(Q,\De,L)$. Here $\De$  
is a non-integrable {\em constraint distribution} over $Q$ of constant rank  that at each $q \in Q$ defines an
$s$-dimensional subspace $\De_q\subset T_qQ$. A curve $q(t)$ on $Q$ satisfies the  nonholonomic constraints
if $\dot q(t)\in \De_{q(t)}$ at all time $t$.
The Lagrangian function  $L:TQ\too\R$  is of the form kinetic energy minus potential
energy, $L=T-V$, where $T$ defines 
a Riemannian metric on $Q$ and  $V:Q\to \R$.

The equations of motion are obtained via the Lagrange-D'Alembert principle of ideal constraints.
In local coordinates, it leads to\footnote{The sum convention over repeated indices is in use.}
\begin{subequations}\label{LdAeqs}
\begin{align}
\frac{d}{dt}\left ( \frac{\partial L}{\partial \dot q^{i}}\right ) -\frac{\partial L}{\partial q^{i}}&=\lambda_{\alpha}\,\mu^{\alpha}_{i}(q),\label{Con-1}
\qquad i=1, \dots, n, 
\\\label{Con-2}
\dot q(t)\in \De_{q(t)}. 
\end{align}
\end{subequations}
In the above equations 
$\mu^\alpha=\mu^\alpha_idq^i$ are any set of $n-s$  independent one-forms on $Q$ whose joint annihilator is $\De$,
and the scalars $\lambda_{\alpha},\hspace{1mm} \alpha=1,...,n-s$, that specify the reaction forces, are 
sometimes called Lagrange multipliers and are
uniquely determined by the constraints \eqref{Con-2}. Since the constraints are  ideal, 
 the energy $E=T+V$  is conserved along the flow. For more details we refer to \cite{Arn,Bl}.

\subsection{Discrete nonholonomic systems}\label{NI}

Discretisations of the Lagrange-d'A\-lem\-bert  prin\-ci\-ple for  nonholonomic systems have been introduced in \cite{CoMa,McPer} as an  extension of variational integrators (see \cite{MarsdenWest,Mose} and references therein for more details in this last topic). According to these references, the discretisation of the nonholonomic system
$(Q,\De,L)$ requires   the construction of  a {\em discrete Lagrangian} and a {\em discrete constraint space} $D_d\subset Q\times Q$. In the following definition we emphasise the role of the time step which is important for our purposes.

\begin{definition}\label{DiscNHSet} 
A discretisation of the  nonholonomic system $(Q,L,\De)$ is given by a time step $\varepsilon>0$ and a pair $(D_d,L^{(\varepsilon)}_d)$ where:
\begin{enumerate}
\item The discrete constraint space $D_d$ is a submanifold of $Q\times Q$ of dimension $2n-s$ with the additional property that
\[
I_d=\left \{ (q,q)\,|\,q\in Q\right \} \subset D_d,
\]
and, moreover, for all $q\in Q$
\begin{equation}
\label{E:ConsistDd}
T_{(q,q)}D_d=\{(v,w) \in T_qQ \times T_qQ \, | \, v-w\in \De_q\}.
\end{equation}
\item $L^{(\varepsilon)}_d:Q\times Q\too\R$ is the discrete Lagrangian.
\end{enumerate}
\end{definition}

\begin{remark}
\label{Rm:DiscLag}
 The discrete Lagrangian in the definition is taken as an approximation of the action functional over a time
interval of length $\varepsilon$, say $L^{(\varepsilon)}_d(q_k,q_{k+1})\simeq \int_{t_k}^{t_k+\varepsilon}L(q(t),\dot q(t))\, dt$, where the curve $q(t)$   satisfies $q(t_k)=q_k, \; q(t_{k}+\varepsilon)=q_{k+1}$, see e.g. \cite{MarsdenWest}.
The dependence of  $L^{(\varepsilon)}_d$ on $\varepsilon$ is not made explicit in    \cite{CoMa},  \cite{FeZen},  \cite{McPer}  and  \cite{FeChap}.
\end{remark}

\begin{remark}
 Condition \eqref{E:ConsistDd} is required so that the discrete constraint space 
$D_d$ is a consistent approximation of the constraint distribution $\De$.  It is   tacitly assumed in  \cite{FeZen} and
a particular case of it is mentioned in \cite{McPer}.
\end{remark}

The {\em discrete Lagrange-D'Alembert principle}  (DLA) defined by Cort\'es and Mart\'inez in \cite{CoMa}
 yields the set of  discrete nonholonomic equations
\begin{subequations}\label{DLA}
\begin{align}
D_1L_d^{(\varepsilon)}(q_k,q_{k+1})&+D_2L^{(\varepsilon)}_d(q_{k-1},q_{k})=\lambda^{(k)}_{\alpha}\,\mu^{\alpha}(q_k),\label{DLAa}\\
(q_k,q_{k+1})\in D_d
\label{DLAb}.
\end{align}
\end{subequations}
Here, just as in \eqref{LdAeqs}, the independent one-forms $\mu^\alpha$ are an arbitrary basis of the annihilator of $\De$.
At each step, the multipliers $\lambda^{(k)}_\alpha$ appearing in \eqref{DLAa} are determined by the condition \eqref{DLAb}.
Under certain technical conditions, the above is a well defined algorithm for the discrete approximation of the solutions of \eqref{LdAeqs}; see \cite{CoMa} for details.
For generalisations of  discrete nonholonomic systems, see \cite{Iglesias}.

\subsection{Finite difference maps and consistency}
\label{SS:DiffMaps}

The performance of a DLA nonholonomic integrator will depend on the choice of the pair $(D_d,L^{(\varepsilon)}_d)$. A possibility to construct them is to use  {\em finite difference  maps} \cite{McPer}.

\begin{definition}\label{D:fdm}\cite{McPer}
A  {\em finite difference  map} $\Psi$ is a diffeomorphism $\Psi : N_0(I_d)\to T_0Q$, where $N_0(I_d)$ is a neighbourhood
of the diagonal $I_d$ in $Q\times Q$ and $T_0Q$ denotes a neighbourhood of the zero section of $TQ$ which satisfies the
following
\begin{enumerate}
\item $\Psi (I_d)$ is the zero section of $TQ$;
\item $\tau_Q \circ \Psi (N_0(I_d))=Q$; and
\item along the diagonal $I_d$ we have
\begin{equation*}
\left . \tau_Q \circ \Psi \right |_{I_d} =\left .\pi_1 \right |_{I_d}=\left . \pi_2  \right |_{I_d},
\end{equation*}
\end{enumerate}
where $\tau_Q:TQ\to Q$ is the bundle projection and $\pi_1$, $\pi_2$ are the projections from $Q\times Q$ to $Q$.
\end{definition}

A simple example of a finite difference map with time step $\varepsilon$ if $Q=\R^n$ is given by
\begin{equation*}
\Psi(q_k,q_{k+1})=\left (q_k, \frac{q_{k+1}-q_k}{\varepsilon} \right ) \in T_{q_k}\R^n.
\end{equation*}

With a finite difference map at hand, one can define the discrete Lagrangian $L_d^{(\varepsilon)}$ by
\begin{equation}
\label{E:LdviaFiniteDifference}
L_d^{(\varepsilon)}=\varepsilon L\circ \Psi,
\end{equation}
that is the simplest quadrature approximation of the action  $\int_{t_k}^{t_k+\varepsilon}L(q(t),\dot q(t))\, dt$.

Following the discussion in \cite{CoMa}, we make the following definition.
\begin{definition}
\label{D:Consistency}
The discrete nonholonomic system defined by the pair $(D_d,L^{(\varepsilon)}_d)$ is {\em consistent with respect to the finite difference map $\Psi$}
if \eqref{E:LdviaFiniteDifference} holds and $\Psi (D_d)\subset \De$.
\end{definition}

In Remark 3.1  of \cite{CoMa} the authors claim that consistency is a necessary condition to guarantee approximation of the
continuous flow  to a desired
level of accuracy. The main contribution of this paper is to examine the validity of this statement by treating in detail a concrete example.
We will examine two different discretisations of the Suslov problem. Only one of them
 is consistent, but they both  lead to discrete algorithms having 
 local truncation errors of the same order. Our results indicate that the statement made in \cite{CoMa} is not entirely correct.

\subsection{Nonholonomic LL systems}\label{LLNI} If the 
configuration space is a Lie group $Q=G$, and both the constraint distribution $\De$ and the Lagrangian $L$ are invariant
under left multiplication on $G$, then we speak  of an {\em LL system}. In this case there is a subspace $\de$  of the Lie algebra $\al=T_eG$, where $e$ is the group identity, such that $\De_g=g\de$ for all $g\in G$. The non-integrability of the constraints means that $\de$ is not a subalgebra.

By left invariance, the Lagrangian $L:TG\too \R$ is of pure kinetic energy and is
determined by its value at the identity. We have $L(g,\dot g)=L(e,g^{-1}\dot g):=\ell(\xi)$, where $\xi =g^{-1}\dot g \in \al$. We call
 $\ell:\al \too \R$  the {\em reduced Lagrangian}. It is defined by the {\em inertia tensor} $\I:\al\too \al^*$
 that specifies  the left invariant kinetic energy metric at the identity. We will label such  LL system by the triple  $(G,\de, \ell)$. 

The reduction of the system by the left action of $G$ leads to the 
 {\em Euler-Poincar\'e-Suslov} equations \cite{FeKoz}
\begin{subequations}\label{E:EuPoSusMom}
\begin{align}
\label{E:EuPoSusMom1}
\dot M&=\ad_{\xi}^*\,M+\lambda_{\alpha}a^{\alpha}, \\
\label{E:EuPoSusMom2}
M &\in \de^*:=\I(\de).
\end{align}
\end{subequations}
Here $M=\I(\xi)$ and $a^\alpha\in \al^*$ are independent vectors whose joint annihilator is $\de$.
As before, the multipliers $\lambda_\alpha$ appearing in \eqref{E:EuPoSusMom1} are uniquely determined
by \eqref{E:EuPoSusMom2}. The above equations are consistent with the Lagrange-D'Alembert principle of ideal constraints and therefore
preserve the energy of the system.

\subsection{Discrete nonholonomic LL systems}\label{LLNII}

The discretisation of  LL systems  in accordance with the DLA algorithm was thoroughly considered by Fedorov and Zenkov
\cite{FeZen}. The authors proposed a discretisation scheme under the natural assumptions that both the 
 discrete Lagrangian $L_d^{(\varepsilon)}:G\times G\too\R$ and the discrete constraint space  $D_d\subset G\times G$ are invariant
 under the diagonal action of $G$ on $G\times G$ by left multiplication. We briefly recall some of their results.
 
 By invariance of $L_d^{(\varepsilon)}$, one can define a {\em reduced discrete Lagrangian} $\ell_d^{(\varepsilon)}:G\too \R$ by the rule 
 \begin{equation*}
L_d^{(\varepsilon)}(g_k,g_{k+1})=L_d^{(\varepsilon)}(e,g_{k}^{-1}g_{k+1}) =: \ell^{(\varepsilon)}_d(W_k)
\end{equation*}
where $W_k:=g_{k}^{-1}g_{k+1}\in G$ is the  {\it left incremental displacement}. One should  interpret $W_k\in G$ an approximation of $\exp{(\varepsilon\xi)}$, where $\xi =g^{-1}\dot g\in \al$. The relation
$g_{k+1}=g_kW_k$,  is   the discrete counterpart of the {\em reconstruction equation} $\dot g=g\xi.$ 
 
 Similarly, by left invariance of $D_d$ there exists a {\em discrete displacement subvariety} $\Ese \subset G$ determined
 by the condition
 \begin{equation*}
(g_k,g_{k+1})\in D_d \qquad \mbox{if and only if} \qquad W_k=g_{k}^{-1}g_{k+1}\in \Ese.
\end{equation*}
Given that $(g,g)\in D_d$ for all $g\in G$, it follows that the identity element $e\in \Ese$. Moreover,
one can easily show, using \eqref{E:ConsistDd}, that $T_e\Ese =\de$. 

Therefore,  under the above invariance conditions on $(D_d,L_d^{(\varepsilon)})$ the corresponding DLA nonholonomic integrator is completely determined by the  pair $(\Ese, \ell^{(\varepsilon)}_d)$. This motivates the following.
\begin{definition}\label{D:LL-disc} 
A left invariant discretisation of the  nonholonomic LL system $(G,\ell,\de)$ is given by a time step $\varepsilon>0$ and a pair $(\Ese,\ell^{(\varepsilon)}_d)$ where:
\begin{enumerate}
\item The discrete displacement subvariety $\Ese$ is a submanifold of $G$ that contains the identity and $T_e\Ese =\de$. In particular, $\dim(\Ese)=\dim (\de)$.
\item The discrete reduced Lagrangian $\ell^{(\varepsilon)}_d:G\too\R$. 
\end{enumerate}
\end{definition}
\begin{remark}
 In accordance with Remark \ref{Rm:DiscLag}, the reduced discrete Lagrangian is  an approximation of the action: 
 $\ell^{(\varepsilon)}_d(W_k)\simeq \int_{t_k}^{t_k+\varepsilon}\ell (g^{-1}(t)\dot g(t))\, dt$, where the curve $g(t)$ satisfies $g(t_k)=e, \; g(t_{k}+\varepsilon)=W_k$.
\end{remark}

Define the {\em discrete Legendre transformation} $ \F\ell_d^{(\varepsilon)}:G\to \al^*$
by the right trivialisation of the derivative of $\ell^{(\varepsilon)}_d$  \cite{BoSu}. For $W\in G$ 
and $\xi \in \al$ we have
\begin{equation}\label{DiscMomentum}
\langle \F\ell_d^{(\varepsilon)}(W),\xi\rangle =
\left . \frac{d}{ds}\right |_{s=0}\ell_d^{(\varepsilon)}(\mbox{exp}(s\xi)W).
\end{equation}

For $W_k\in \Ese$ we define the associated discrete momentum\footnote{This definition is consistent
with \cite{FeZen} but varies slightly from others, like \cite{BoSu}.}
\begin{equation*}
M_k:= \F\ell_d^{(\varepsilon)}(W_k).
\end{equation*}
One should understand $\F\ell_d$ as an approximation of the inertia tensor $\I:\al \to \al^*$. 
Note that numerous complications arise in the discrete setting. Firstly, $ \F\ell_d^{(\varepsilon)}$ is  locally invertible
but in general will fail to be globally invertible (see \cite{FeZen} and the discussion in  \S\ref{S:AsympExpSuslovMoserVes}). 
Secondly, $ \F\ell_d^{(\varepsilon)}$ is a nonlinear map, in fact its
domain is not even a linear space. 

The {\em discrete momentum locus} is defined in  \cite{FeZen} as
\begin{equation*}\label{Locus}
\mathfrak{u}_{\varepsilon}:=\F \ell_d^{(\varepsilon)}(\Ese)\subset \al^*.
\end{equation*}
The discrete momentum locus is a nonlinear subvariety of $\al^*$ that approximates the linear constraint space $\de^*:=\I(\de)\subset \al^*$.
Provided that $\ell_d^{(\varepsilon)}$ is a consistent approximation of the continuous reduced action, it will contain $0\in \al^*$ and will be tangent to $\de^*$.

 Fedorov and Zenkov \cite{FeZen} prove that the discrete nonholonomic equations \eqref{DLA} reduce to the   {\em discrete Euler-Poincar\'e-Suslov equations}   
\begin{subequations}\label{DEPSAlg}
\begin{align}\label{DEPSAlga}
&M_{k+1}=\Ad^*_{W_k}M_k + \lambda_{\alpha}^{(k)}\,a^{\alpha},\\ 
&M_{k+1}\in \mathfrak{u_\varepsilon},
\label{DEPSAlgb}
\end{align}
\end{subequations}
where $W_k\in \Ese$. 
As usual, the multipliers $\lambda_{\alpha}^{(k)}$ are determined by \eqref{DEPSAlgb}.
Since the discrete Legendre transform is only locally invertible, this scheme will generally
be multivalued and some care should be taken in the choice of the branch to adequately approximate the solutions
of \eqref{E:EuPoSusMom} (see the discussion in \cite{FeZen}).

\subsection{Left invariant finite difference maps}\label{SS:LeftInvDiffMaps}

Given that the discretisation of LL systems proposed by Fedorov, Zenkov in \cite{FeZen} considers a discrete Lagrangian
$L_d$ and a discrete constraint space $D_d$ that are invariant with respect to the diagonal action of $G$ on $G\times G$
by left multiplication, it is natural to consider the construction of these objects using  a finite difference map $\Psi$ that 
is equivariant.
Namely, one that satisfies
\begin{equation*}
\label{E:FiniteDiffinvariance1}
\Psi(hg_k,hg_{k+1})=Tl_h \Psi(g_k,g_{k+1}) \qquad \mbox{for all} \qquad h\in G,
\end{equation*}
where $l_h:G\too G$ is left multiplication by $h$.
If in addition to equivariance, the finite difference map $\Psi$ satisfies 
\begin{equation}
\label{E:FiniteDiffinvariance2}
\Psi(g_k,g_{k+1})\in T_{g_k}G, 
\end{equation}
for some $g_k$ (and hence all $g_k\in G$)
then $\Psi$ induces a {\em reduced difference map } defined by
\begin{equation}
\label{E:RedFiniteDiff}
\psi (W_k) =\Psi(e,g_k^{-1}g_{k+1}) \in \al,
\end{equation}
where, as before, $W_k=g_k^{-1}g_{k+1}$.
Note that $\psi$ is defined for $W_k$ on a neighbourhood of the of the identity $e\in G$ and maps into the 
Lie algebra. Moreover, we have $\psi (e)=0$.
We now state the following proposition whose proof is a direct consequence of the definitions.

\begin{proposition}
\label{P:Consistency} Let $(\Ese, \ell_d^{(\varepsilon)})$ be the invariant discretisation of the LL system $(G,\de,\ell)$ that arises as the
reduction of the invariant pair $(D_d,  L_d^{(\varepsilon)})$, where 
 $D_d\subset G\times G$ and $L_d^{(\varepsilon)}:G\times G\too \R$.
 Let $\Psi$ be an equivariant difference map that satisfies \eqref{E:FiniteDiffinvariance2}.
 Then, the discretisation of the unreduced LL system  $(G,\De,L)$
 defined by  $( D_d, L_d^{(\varepsilon)})$ is
  consistent with respect to  $\Psi$ if and only if the induced 
reduced difference map $\psi$ defined by \eqref{E:RedFiniteDiff} satisfies
\begin{equation}
\label{E:ReducedConsistency}
\ell_d^{(\varepsilon)}=\varepsilon \ell \circ \psi, \qquad \psi (\Ese)\subset \de.
\end{equation}
\end{proposition}

In practice, to construct an invariant consistent discretisation of an LL system one can start with a {\em retraction map},  that is a local diffeomorphism
 $\tau:\al\to G$, with the property that $\tau (0)=e$ and
$T_0\tau=\mbox{id}_\al$.  The reduced finite difference map $\psi$ can be taken as the local  inverse of $\tau$ (defined in a neighbourhood 
$N_e$ of $e\in G$), and the discrete displacement subvariety as $\Ese:=\tau(\de)\cap N_e$. Examples of retraction maps are the 
exponential  and  the Cayley maps.

\section{The Suslov problem}\label{SusPro}

The Suslov problem, first introduced in \cite{Suslov}, is a prototype example of a nonholonomic LL system in which the Lie group $G=SO(3)$.
Physically it models  the motion of a rigid body under  its own inertia subject to the nonholonomic
constraint that one of the components of the angular velocity as seen in the {\em body frame} vanishes. 
Here we recall a series of
known facts of the problem and we give a Taylor expression of its reduced solutions in order to later examine the performance
of geometric integrators for the system.

\subsection{Euler-Poincar\'e-Suslov equations}
Without loss of generality, we  assume that
the body frame has been chosen in such way that the nonholonomic  constraint is $\omega_3=0$,
where $\omega\in \R^3$ denotes the angular velocity vector written the body  frame.
As usual, we interpret $\omega$ as an element of the Lie algebra $\al=\R^3$ equipped with the vector product $\times$.
  We have  $\widehat \omega = g^{-1}\dot g$, where $g\in SO(3)$ is the attitude matrix of the body and $\widehat \omega$ is the  
skew-symmetric matrix that represents  $\omega\in \R^3$ via the  hat map, see \eqref{Isom} below.

The nonholonomic constraint defines a left-invariant non-integrable, rank 2 distribution $\De\subset TSO(3)$ determined at the identity by the linear subspace 
\begin{equation}
\label{E:deSuslov}
\de =\left \{ \omega \in  \R^3 \, | \, \omega_3=0 \right \}.
\end{equation}
 It is clear that  $\de$ is not a subalgebra of $(\R^3,\times)$.

The reduced   Lagrangian of the system is $\ell(\omega) = \frac{1}{2} \langle \I \omega, \omega\rangle$
where $\I$ is the inertia tensor of the body and  $\langle \cdot , \cdot \rangle$ is the euclidean product in $\R^3$. Rotating the body frame about its third axis if necessary, we can assume without loss of generality that 
\begin{equation*}\label{InertiaTensor}
\I = \left ( \begin{array}{ccc} I_{11} & 0 & I_{13} \\ 0 & I_{22} & I_{23} \\  I_{13} & I_{23} & I_{33} \end{array} \right ).
\end{equation*}
The Euler-Poincar\'e-Suslov equations \eqref{E:EuPoSusMom} for the angular momentum vector $M=\I\omega$ are 
 \begin{equation}
 \label{eq:EPS-M}
\dot M = M \times  ( \I^{-1}M)  +\lambda e_3.
\end{equation}
where $\lambda$ is a multiplier that is uniquely determined by the condition that $M\in \de^*:=\I(\de)$, and $e_3=(0,0,1)$.
 Explicitly we have
\begin{equation}
\label{E:defdestarSuslov}
\de^*=\left \{ M\in \R^3 \,  | \, M_3=\frac{ I_{13}I_{22}M_1+ I_{11}I_{23} M_2}{I_{11}I_{22}} \right \}.
\end{equation}
Therefore, the first two components of \eqref{eq:EPS-M} become
\begin{equation}\label{E:EulerPoincareDual}
\begin{split}
\dot M_1&=-\frac{M_2}{I_{11} I_{22}^2} \left ( I_{13}I_{22}M_1+ I_{11}I_{23} M_2\right ),\\
\dot M_2&=\,\,\,\,\,\frac{M_1}{I_{22} I_{11}^2} \left ( I_{13}I_{22}M_1+ I_{11}I_{23} M_2\right ).
\end{split}
\end{equation}

 The above equations
 preserve the restriction of the energy $\frac{1}{2}\langle \I^{-1}M,M \rangle$ to the constraint space $\de^*$. Explicitly, the first integral is
\begin{equation}
\label{E:EnergySuslov}
 E_c(M_1,M_2) =\frac{I_{22}M_1^2+I_{11}M_2^2}{2I_{11}I_{22}}.
\end{equation}
Using this first integral, equations \eqref{E:EulerPoincareDual} can be integrated explicitly in terms of hyperbolic functions.
\subsection{Taylor expansion of the solutions} In order to evaluate the order of local truncation error of the geometric integrators
for the Suslov problem to be defined ahead, we perform a Taylor expansion of the solutions 
of \eqref{E:EulerPoincareDual}. By repeated differentiation of these equations we find
\begin{equation}\label{E:ExpansionMcont}
\begin{split}
M_1&(t+\varepsilon)=M_1(t) + \varepsilon \mu_1(t) +  \varepsilon^2 \mu_2(t) +  \varepsilon^3 \mu^{(C)}_3(t) +\mathcal{O}(\varepsilon^4), \\
M_2&(t+\varepsilon)=M_2(t) + \varepsilon \nu_1(t) +  \varepsilon^2 \nu_2(t) +  \varepsilon^3 \nu^{(C)}_3(t) +\mathcal{O}(\varepsilon^4),
\end{split}
\end{equation}
with
\begin{equation}
\begin{split}
\label{E:ExpansionMAuxGeneral}
\mu_1&= - \frac{M_2(I_{13}I_{22}M_1 +I_{11}I_{23} M_2)}{I_{22}^2I_{11}} ,  \\ 
\mu_2 &= -\frac{\left ( I_{13}I_{22}M_1+I_{11}I_{23}M_2\right ) \left ( I_{13}I_{22}M_1^2+2I_{11}I_{23}M_1M_2-I_{11}I_{13}M_2^2\right )}{2I_{11}^3I_{22}^3}, \\
\nu_1&=\frac{M_1(I_{13}I_{22}M_1 +I_{11}I_{23} M_2)}{I_{22}I_{11}^2}, \\
\nu_2&=\frac{\left ( I_{13}I_{22}M_1+I_{11}I_{23}M_2\right ) \left ( I_{23}I_{22}M_1^2-2I_{22}I_{13}M_1M_2-I_{11}I_{23}M_2^2\right )}{2I_{11}^3I_{22}^3},
\end{split}
\end{equation}
and
\begin{equation}
\begin{split}
\label{E:ExpansionMAuxCont}
&\mu^{(C)}_3 = \frac{-1}{12I_{11}^4I_{22}^5}\left ( I_{13}I_{22}M_1+I_{11}I_{23}M_2\right ) \left(      -9I_{11}I_{13}I_{22}I_{23}M_1M_2^2 -2I_{11}^2I_{23}^2M_2^3  \right. \\  &\left. 
 \qquad \qquad  + 3I_{13}I_{22}^2I_{23}M_1^3  + I_{22}(4I_{11}I_{23}^2-5I_{13}^2I_{22}) M_1^2M_2+I_{11}I_{13}^2I_{22}M_2^3 \right), \\
&\nu^{(C)}_3= \frac{1}{12I_{11}^5I_{22}^4}\left ( I_{13}I_{22}M_1+I_{11}I_{23}M_2\right ) \left(  -9I_{11}I_{13}I_{22}I_{23}M_1^2M_2  -2I_{22}I_{13}^2M_1^3  \right. \\ &\left.  \qquad \qquad
 + 3I_{13}I_{11}^2I_{23}M_2^3  + I_{11}(4I_{22}I_{13}^2-5I_{23}^2I_{11}) M_1M_2^2+I_{22}I_{23}^2I_{11}M_1^3 \right).
\end{split}
\end{equation}

\subsection{Expressions in $\so(3)$}

For our treatment ahead it is useful to write the angular velocity $\omega$, the angular momentum $M$ and the reduced Lagrangian $\ell$, in terms of skew-symmetric $3\times 3$ matrices. 

Recall (see e.g. \cite{MaRa}) that the {\em hat map} is the Lie algebra isomorphism $\;\widehat{} :\R^3\to \so(3)$ defined by
\begin{equation}\label{Isom}
\widehat \omega=\left (
\begin{array}{ccc}
0&-\omega_3&\omega_2\\
\omega_3&0&-\omega_1\\
-\omega_2&\omega_1&0
\end{array}
\right ) \in\so(3).
\end{equation}
The hat map is also an isometry between $\R^3$ with the standard Euclidean metric  and  $\so(3)$  equipped 
with the scalar product 
\begin{equation}
\label{E:pairing}
\langle \xi , \eta \rangle = \frac{1}{2}\mbox{Trace}(\xi \eta^T), \qquad \xi, \eta \in \so(3). 
\end{equation}

A short calculation shows that 
\begin{equation*}
\widehat M = \widehat{\I \omega}= \mathbb{J}\widehat \omega+ \widehat \omega \mathbb{J},
\end{equation*}
where the symmetric matrix $\mathbb{J}$ is given by
\begin{equation*}
\J = \left ( \begin{array}{ccc} \frac{1}{2}( I_{22}+ I_{33}-I_{11}) & 0 & -I_{13} \\ 0 & \frac{1}{2}( I_{11}+ I_{33}-I_{22}) & - I_{23} \\ - I_{13} & - I_{23} &  \frac{1}{2}( I_{11}+ I_{22}-I_{33}) \end{array} \right ).
\end{equation*}
Using the properties of the trace and the skew-symmetry of $\widehat \omega$ we can write 
\begin{equation}
\label{E:ContLag}
\ell(\omega) = \frac{1}{2}\mbox{Trace}(\mathbb{J}\widehat  \omega \widehat \omega^T).
\end{equation}
In the following sections we will make an indistinctive treatment of vectors in $\R^3$ and skew-symmetric matrices in $\so(3)$.
There should be no risk of confusion and any kind ambiguities are resolved by the hat map.

\section{Two different discretisations of the Suslov problem}
\label{S:Pade}

We now define two different invariant discretisations of the Suslov problem that follow the prescription of Fedorov-Zenkov \cite{FeZen} described in 
 \S\ref{LLNII}. Recall that such discretisations are determined by a choice of discrete displacement subvariety $\Ese\subset SO(3)$ 
 and a reduced  discrete Lagrangian $\ell_d^{(\varepsilon)}:SO(3)\too \R^3$.

The choice of $\Ese$ for both discretisations is given in \S\ref{S:S}. The two different choices of the reduced discrete Lagrangian are respectively given 
in  \S\ref{S:l1} and  \S\ref{S:linfty}. We show that only the second discretisation is consistent in the sense of Definition \ref{D:Consistency}.
 
Our construction will use the {\em Cayley map} (see e.g. \cite{Iserles}).  For a skew-symmetric matrix $\omega \in \so(3)$ define\footnote{The factor of the time step $\varepsilon$ is included for consistency with physical units.}
\begin{equation}
\label{E:CayleyGeneral}
\mbox{Cay}_\varepsilon(\omega):=\left (e+\frac{\varepsilon\omega}{2} \right )\left (e-\frac{\varepsilon\omega}{2} \right )^{-1} \in SO(3).
\end{equation}
The Cayley map is injective and for 
 $W\in \mbox{range}(\mbox{Cay}_\varepsilon)\subset SO(3)$ , we
have
\begin{equation}
\label{E:CayInv}
\mbox{Cay}_\varepsilon^{-1}(W)=\frac{2}{\varepsilon}(W-e)(W+e)^{-1}.
\end{equation}
By Euler's theorem (see e.g. \cite{MaRa}) we know that, except for the identity, any matrix in $SO(3)$ is a rotation through an angle $\theta$ about a certain axis.
Such matrix will have eigenvalues $1, e^{i\theta}, e^{-i\theta}$.
On the other hand, from the above formula we see that the range of $\mbox{Cay}_\varepsilon$ consists of matrices in $SO(3)$ that do not have eigenvalue $-1$. 
Therefore, the range of $\mbox{Cay}_\varepsilon$ consists of those matrices whose angle of rotation $\theta$   lies strictly between $-\pi$ and $\pi$.

\subsection{The discrete  constraint displacement subvariety} \label{S:S} We define $\Ese\subset SO(3)$ by
\begin{equation}
\label{eq:DefS}
\mathcal{S}:=\mbox{Cay}_\varepsilon (\mathfrak{d}),
\end{equation}
where $\mathfrak{d}$ is given by \eqref{E:deSuslov}.  

Setting $\omega_3=0$ in \eqref{Isom}, and using \eqref{E:CayleyGeneral}, we obtain the following expression for the matrix $\mbox{Cay}_\varepsilon(\omega_1, \omega_2,0)$
\begin{equation}
\label{E:Cayd}
\frac{1}{4+\varepsilon^2(\omega_1^2+\omega_2^2)}\left ( \begin{array}{ccc}
4+\varepsilon^2(\omega_1^2-\omega_2^2) & 2\varepsilon^2\omega_2\omega_1 & 4\varepsilon\omega_2\\
2\varepsilon^2\omega_2\omega_1 & 4-\varepsilon^2(\omega_1^2-\omega_2^2) & -4\varepsilon\omega_1\\
 -4\varepsilon\omega_2 & 4\varepsilon\omega_1 & 4-\varepsilon^2(\omega_1^2+\omega_2^2)
\end{array}
\right ).
\end{equation}
Matrices of the above form with $(\omega_1, \omega_2)\in \R^2$  parametrise $\Ese$.
Note that the above matrix has
axis of rotation $(\omega_1,\omega_2,0)$. Therefore, the condition $W_k\in \Ese$ can be restated by saying  that  the incremental displacements
 $W_k=g_k^{-1}g_{k+1}$ have axis of rotation 
 perpendicular to $e_3$, and angle of rotation $-\pi< \theta < \pi$.

Using \eqref{E:Cayd} one can check that the condition $T_e\Ese =\de$ in Definition \ref{D:LL-disc} is satisfied.

\subsection{Definition of the non-consistent discretisation.}\label{S:l1}

As first considered by  Moser and Veselov in  \cite{Mose}, we define the reduced discrete Lagrangian $\ell^{(1,\varepsilon)}:SO(3)\too \R$  by
\begin{equation}
\label{E:DiscLagMosVes}
\ell^{(1,\varepsilon)}_d(W)=-\frac{1}{\varepsilon}\mbox{Trace}(\J W).
\end{equation}
Up to the addition of an irrelevant
constant, it is obtained by formally  putting  $\widehat \omega = \frac{1}{\varepsilon}(W-e)$ in \eqref{E:ContLag}.
The superscript $1$ in the notation indicates that $ \frac{1}{\varepsilon}(W-e)$  is the first term in the expansion of \eqref{E:CayInv} as a power series in $W-e$. 

\begin{proposition}
\label{P:NonConsistent}
The geometric discretisation of the Suslov problem defined by the constraint subvariety $\Ese$ given by \eqref{eq:DefS} and the
reduced discrete Lagrangian \eqref{E:DiscLagMosVes} is not consistent in the sense of Definition \ref{D:Consistency}. \end{proposition}
\begin{proof}
Up to an irrelevant constant term we have $\ell^{(1,\varepsilon)}_d=\varepsilon \ell \circ \sigma_1$ where  $\sigma_1$ is defined by 
\begin{equation*}
 \sigma_1(W)=\frac{1}{\varepsilon}(W-e).
\end{equation*}
In view of Proposition \ref{P:Consistency} and equation \eqref{E:ReducedConsistency}, the discretisation is consistent
if there exists an open submanifold $\Ese'$ of $\Ese = \mbox{Cay}_\varepsilon(\de)$ containing the identity, such that 
$\sigma_1(\Ese')\subset \de$.
But this is not possible since for $W\in \Ese$, $\sigma_1(W)\in \so(3)$ if only if $W=e$.
\end{proof}

In  \cite{FeZen} Fedorov and Zenkov considered a geometric discretisation of the Suslov problem having discrete reduced Lagrangian  given by
\eqref{E:DiscLagMosVes} and with constraint displacement subvariety $\Ese':=\exp(\de)$. Such choice of $\Ese'$ consists of all  matrices in $SO(3)$ having
axis of rotation  perpendicular to $e_3$. Therefore, $\Ese$ given by \eqref{eq:DefS} is an open and dense subset of $\Ese'$ that contains the identity. The complement $\Ese'\setminus \Ese$ consists of matrices in  $SO(3)$
 that are a rotation by $\pi$ about an axis that is perpendicular to $e_3$. In particular, given that $\Ese$ and $\Ese'$ coincide on a neighbourhood of the identity,
 we conclude that the geometric discretisation that we have introduced above is equivalent to the one considered by Fedorov and Zenkov.

\begin{corollary}
The geometric discretisation of the Suslov problem considered by Fedorov and Zenkov in  \cite{FeZen} is not consistent in the sense of Definition \ref{D:Consistency}.
\end{corollary}

It is remarkable that Fedorov and Zenkov \cite{FeZen} do not seem to notice that their discretisation is not consistent given that  they explicitly mention the consistency condition at the end of Section 2 in their paper. Having
chosen $\Ese'=\exp(\de)$, a way to construct a consistent discretisation is to define  $\ell^{(\varepsilon)}_d=\varepsilon \ell \circ \log$, where $\log$ is a local inverse of the exponential map in a neighbourhood of $0\in \so(3)$. 
The inverse Cayley map allows us to make an alternative consistent construction without dealing with the matrix logarithm. This is explained next.

\subsection{Definition of the consistent discretisation.}\label{S:linfty} Consider the reduced constrained Lagrangian $\ell^{(\infty,\varepsilon)}_d:=\varepsilon \ell \circ \mbox{Cay}_\varepsilon^{-1}$, where $\ell$ is given by
\eqref{E:ContLag}. Explicitly, for $W$ in the  range of $\mbox{Cay}_\varepsilon^{-1}$ we have 
 \begin{equation}
 \label{E:DiscLagCayley}
 \begin{split}
\ell^{(\infty,\varepsilon)}_d(W)=\frac{2}{\varepsilon}\mbox{Trace}\left (\J(2 -W-W^T)(2+W+W^T)^{-1} \right ).
\end{split}
\end{equation}
It is immediate to check that the geometric discretisation of the Suslov problem defined by $\Ese$ and $\ell^{(\infty,\varepsilon)}_d$ is consistent. The
role of $\psi$ in \eqref{E:ReducedConsistency} is played by $\mbox{Cay}_\varepsilon^{-1}$.

The $\infty$ superscript in the notation indicates that we use the exact value of $\mbox{Cay}_\varepsilon^{-1}$ in the definition of $\ell^{(\infty,\varepsilon)}_d$ instead of a 
partial sum approximation in its expansion as a  power series in $W-e$ (compare with the definition of $\ell^{(1,\varepsilon)}_d$).

\section{Analysis of the  non-consistent discretisation $(\Ese,\ell^{(1,\varepsilon)}_d)$}
\label{S:req1}

As mentioned before, this is a reparametrisation of the discretisation of the Suslov problem considered by Fedorov and Zenkov in \cite{FeZen}.
We will determine the order of local truncation error of the approximation of the continuous flow by the resulting discrete Euler-Poincar\'e-Suslov equations.
To avoid double subscripts, throughout this section we write $\omega_1=u, \, \omega_2=v$.

\subsection{The momentum locus $\mathfrak{u}^{(1)}_{\varepsilon}$}
\label{E:MomLocusMosVes}
We identify $\so(3)^*$ with $\so(3)$ via the inner product \eqref{E:pairing}. The discrete Legendre transformation \eqref{DiscMomentum}
associated to the discrete Lagrangian $\ell^{(1,\varepsilon)}$ given by \eqref{E:DiscLagMosVes} is computed to be 
\begin{equation*}
\label{E:DiscLegTrans1}
\mathbb{F}\ell^{(1,\varepsilon)}_d(W)=\frac{1}{\varepsilon}(W\J-\J W^T).
\end{equation*}
Using   \eqref{E:Cayd}, we find that for $W=\mbox{Cay}_\epsilon(u, v,0)$ we have
$M=\F \ell_d^{(1,\varepsilon)}(W) $ given by
\begin{subequations}\label{MParame}
\begin{align}
M_1=&\frac{2}{4+\varepsilon^2(u^2+v^2)}(2I_{11}u +\varepsilon v (I_{13}u+I_{23}v)),\label{MParamea}\\
M_2=&\frac{2}{4+\varepsilon^2(u^2+v^2)}(2I_{22}v-\varepsilon u(I_{13}u+I_{23}v)),\label{MParameb}\\
M_3=&\frac{2}{4+\varepsilon^2(u^2+v^2)}(2(I_{13}u+I_{23}v)+\varepsilon(I_{22}-I_{11})uv),\label{MParamec}
\end{align}
\end{subequations} 
where we are using the hat map to identify $\so(3)^*\cong \so(3)$ with $\R^3$.

By putting $(u,v)=(0,0)$ in  \eqref{MParame} it is clear that $0\in \mathfrak{u}^{(1)}_\varepsilon$. Moreover, a direct calculation using this parametrisation shows that a
normal vector to $\mathfrak{u}^{(1)}_{\varepsilon}$ at the origin is $(I_{22}I_{13},I_{11}I_{23},-I_{11}I_{22})$. This vector is readily seen to be the normal vector
to the plane $\mathfrak{d}^*$ defined in \eqref{E:defdestarSuslov}. Therefore we have shown the following.
\begin{proposition}
For any $\varepsilon>0$, the momentum locus $\mathfrak{u}^{(1)}_{\varepsilon}$ and the plane $\mathfrak{d}^*$ are tangent at the origin in $\so(3)^*$.
\end{proposition}
Figure \ref{F:MomLocusMosVes} illustrates both the momentum locus $\mathfrak{u}^{(1)}_{\varepsilon}$ and the plane $\de^*$ for generic 
numerical values of $\I$ and a small value of $\varepsilon$.
\begin{figure}[h!]
\centering
\includegraphics[width=6.5cm]{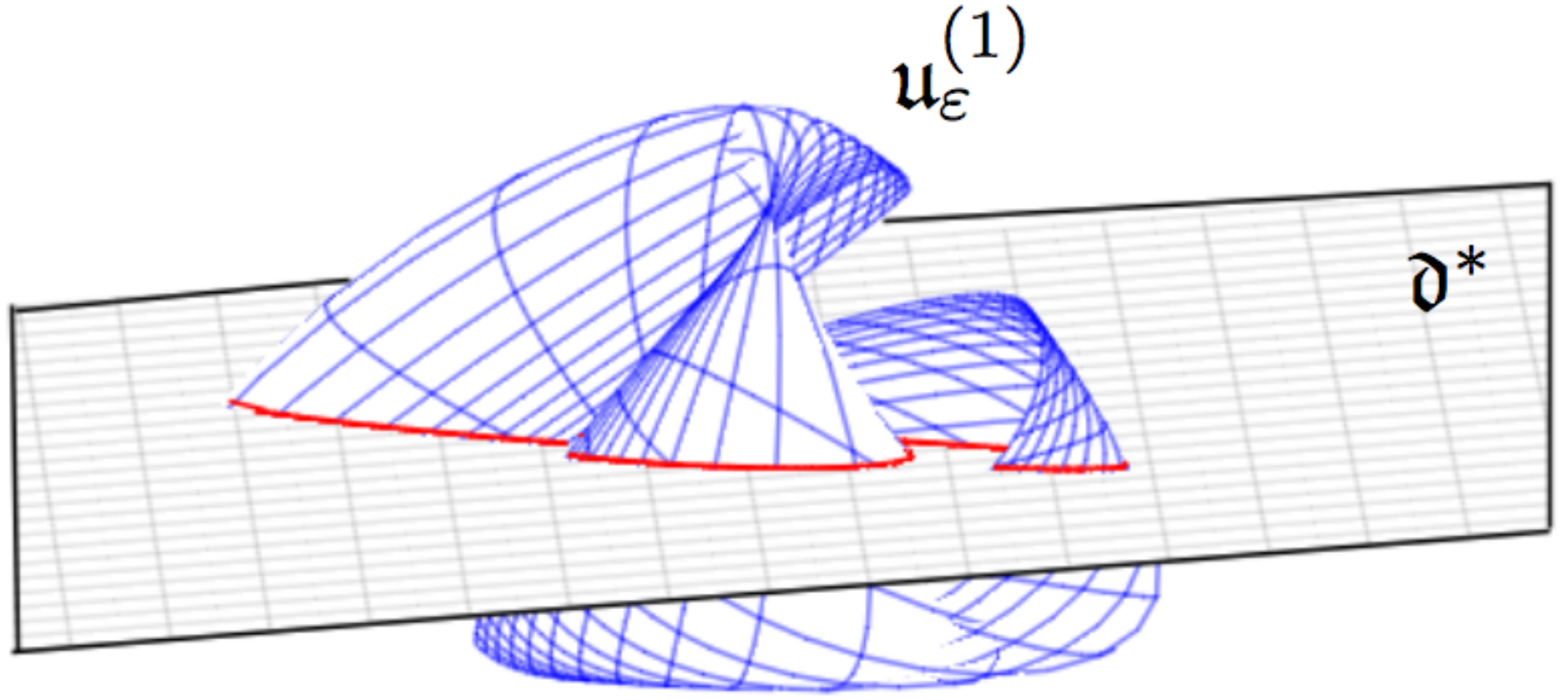}
\caption{{\small The discrete momentum locus $\mathfrak{u}^{(1)}_{\varepsilon}$ defined by $\ell_d^{(1,\varepsilon)}$ and the plane $\de^*$ immersed in $\so(3)^*=\R^3$. Although it cannot be appreciated
from the figure, the surface is tangent to the plane at the origin.
}}
\label{F:MomLocusMosVes}
\end{figure}

The figure illustrates that $\mathfrak{u}^{(1)}_{\varepsilon}$ is bounded, which can be
directly proved from  \eqref{MParame}. It is a complicated surface with pinch points and self intersections. In the Appendix we show that  $\mathfrak{u}^{(1)}_{\varepsilon}$  is contained in the zero locus of a degree $4$ polynomial
in $M_1,M_2,M_3$.

\subsection{The discrete  Euler-Poincar\'e-Suslov equations}
\label{S:DiscreteEPSMoserVes}

Now, we turn our attention towards the discrete Euler-Poincar\'e-Suslov equations \eqref{DEPSAlg}, which in our case read
\begin{equation}\label{MdiscGene}
M_{k+1}=W_k^T\,M_k\,W_k+\lambda_k\left ( \begin{array}{ccc}
0&-1&0\\
1&0&0\\
0&0&0
\end{array}
\right )
\end{equation}
where $W_k\in\mathcal{S}$ and the multiplier $\lambda_k$ is determined by the condition that $M_{k+1}\in \mathfrak{u}^{(1)}_\varepsilon \subset \so(3)^*\cong \so (3)$.

Using  \eqref{E:Cayd} and the hat map,  the above identity between skew-symmetric $3\times 3$ matrices is written in vector form as
\begin{equation*}\label{Mdisc}
 M_{k+1}= M_k+\frac{4\varepsilon}{4+\varepsilon^2(u_k^2+v_k^2)}\left (  \begin{array}{c}
-v_k(I_{13}u_k+I_{23}v_k)\\
\,\,\,\,\,\,\,u_k(I_{13}u_k+I_{23}v_k)\\
\,\,\,\,(I_{11}-I_{22})\,u_kv_k
\end{array}\right ) +\left ( \begin{array}{c}
0\\
0\\
\lambda_k
\end{array}\right ).
\end{equation*}
Taking into account 
 \eqref{MParame}, the first two components in the above equation yield
 \begin{equation}
\label{E:F-Z-scheme-full}
\begin{split}
\frac{4I_{11}u_{k+1}+2\varepsilon v_{k+1}(I_{13}u_{k+1}+ I_{23}v_{k+1})}{4+\varepsilon^2(u_{k+1}^2+v_{k+1}^2)} &=
\frac{4I_{11}u_{k}-2\varepsilon v_{k}(I_{13}u_{k}+I_{23}v_{k})}{4+\varepsilon^2(u_{k}^2+v_{k}^2)} \, ,\\
\frac{4I_{22}v_{k+1}-2\varepsilon u_{k+1}(I_{23}v_{k+1}+I_{13}u_{k+1})}{4+\varepsilon^2(u_{k+1}^2+v_{k+1}^2)} &=\frac{4I_{22}v_{k}+
2\varepsilon u_{k}(I_{23}v_{k}+I_{13}u_{k})}{4+\varepsilon^2(u_{k}^2+v_{k}^2)} \, .
\end{split}
\end{equation}

Given $(u_{k},v_{k})$, in order to obtain $(u_{k+1},v_{k+1})$ one needs to solve the polynomial equations
\begin{equation}
\label{E:F-Z-scheme}
p_2(u,v)=0, \qquad q_2(u,v)=0,
\end{equation}
for $u$ and $v$ where $p_2$, $q_2$ are degree two polynomials  given by
\begin{equation*}
\begin{split}
p_2(u,v)&=4I_{11}u+2\varepsilon I_{13}uv+2\varepsilon I_{23}v^2-A(4+\varepsilon^2(u^2+v^2)), \\
q_2(u,v)&=4I_{22}v-2\varepsilon I_{23}uv-2\varepsilon I_{13}u^2-B(4+\varepsilon^2(u^2+v^2)).
\end{split}
\end{equation*}
Here $A$ and $B$ depend on $(u_{k},v_{k})$ as
\begin{equation*}
A=\frac{4I_{11}u_{k}-2\varepsilon v_{k}( I_{13}u_{k}+ I_{23}v_{k})}{4+\varepsilon^2(u_{k}^2+v_{k}^2)}, \qquad 
B=\frac{4I_{22}v_{k}+2\varepsilon u_{k} ( I_{23}v_{k}+I_{13}u_{k})}{4+\varepsilon^2(u_{k}^2+v_{k}^2)} .
\end{equation*}

In order to understand the number of solutions that \eqref{E:F-Z-scheme} has, we compute the resultant of the polynomials $p_2, q_2$ with respect
to the variable $u$. This gives a polynomial of degree four in $v$ whose 
leading coefficient  is
\begin{equation*}
4\varepsilon^5 (I_{13}^2+I_{23}^2)(2I_{13}B-2I_{23}A+\varepsilon( A^2+B^2)).
\end{equation*}
Under the assumption that the inertia tensor is not diagonal, for generic $(u_{k},v_{k})$, this is non-zero. Hence \eqref{E:F-Z-scheme} generically admits 4 different (possibly complex) solutions for $(u,v)$.
 Experiments show that if $u_{k},v_{k}$ are real, and $\varepsilon>0$ is small enough, then two solutions are real and two are complex. This is consistent with the results reported by Fedorov-Zenkov in \cite{FeZen}, and with the degree of the polynomial $p_4$ given in Proposition \ref{4orderp} in the Appendix.
\subsection{Approximation of the continuous flow}
\label{S:AsympExpSuslovMoserVes}

From the above discussion it follows that \eqref{E:F-Z-scheme-full} 
generically defines a 4-valued map from $\C^2$ to itself. We shall now explain how to select a branch  to construct
a (locally defined) single valued discrete time map\footnote{The subindices  $1$, $\varepsilon$,
on $\mathcal{B}_{(1,\varepsilon)}$ are inherited from the  discrete Lagrangian $\ell_d^{(1,\varepsilon)}$.}  $\mathcal{B}_{(1,\varepsilon)}$ on $\mathcal{S}$ that approximates
the flow of the continuous Suslov problem on $SO(3)$.

A matrix $W_k\in \mathcal{S}$ specifies a point $(u_k,v_k)\in \R^2$ through the inverse Cayley map ($W_k$ is 
represented by \eqref{E:Cayd} replacing $(\omega_1,\omega_2)$ by $(u_k,v_k)$).
Applying the implicit function theorem to  equations  \eqref{E:F-Z-scheme-full} one can show that, for $\varepsilon$ small enough,
there exists a unique solution for $(u_{k+1},v_{k+1})$ that depends smoothly on $u_k, v_k$ and $\varepsilon$ and converges to
$(u_k,v_k)$ when $\varepsilon \too 0$.  The point  $(u_{k+1},v_{k+1})$ defines the matrix $W_{k+1}:=\mathcal{B}_{(1,\varepsilon)}(W_k)\in \mathcal{S}$ via the Cayley map  \eqref{E:Cayd}. The asymptotic expansion for  $(u_{k+1},v_{k+1})$ in terms of $(u_k,v_k)$  is given by
\begin{equation}
\label{E:AsymMosVesB}
\begin{split}
u_{k+1}&=u_k + \varepsilon A_1 +\varepsilon^2 A_2+\varepsilon^3A_3^{(1)}  +\mathcal{O}(\varepsilon^4), \\
v_{k+1}&=v_k  + \varepsilon B_1 +\varepsilon^2 B_2+\varepsilon^3B_3^{(1)}  +\mathcal{O}(\varepsilon^4),
\end{split}
\end{equation}
where
\begin{equation}
\label{E:AsymAuxGeneral}
\begin{split}
A_1&= - \frac{v_k(I_{13}u_k+I_{23}v_k)}{I_{11}}, \\
A_2 &=  -\frac{1}{2I_{11}^2I_{22}}(I_{13}u_k+I_{23}v_k)(I_{11}I_{13}u_k^2+2I_{23}I_{11}u_kv_k-I_{22}I_{13}v_k^2), \\
B_1&= \frac{u_k(I_{13}u_k+I_{23}v_k)}{I_{22}}, \\
B_2 &= -\frac{1}{2I_{11}I_{22}^2}(I_{13}u_k+I_{23}v_k)(I_{22}I_{23}v_k^2+2I_{13}I_{22}u_kv_k-I_{11}I_{23}u_k^2),
\end{split}
\end{equation}
and
\begin{equation*}
\label{E:AsymAuxMosVes}
\begin{split}
A_3^{(1)}&= \frac{-1}{4I_{11}^3I_{22}^2}(I_{13}u_k+I_{23}v_k)\left( 3I_{11}^2I_{13}I_{23}u^3_k-9I_{11}I_{13}I_{22}I_{23}u_kv^2_k\right.\\
&\left.+I_{11}(2I_{11}I_{22}^2-2I_{11}^2I_{22}+4I_{11}I_{23}^2-5I_{13}^2I_{22})u^2_kv_k -I_{22}(2I_{11}I_{23}^2-I_{13}^2I_{22})v^3_k\right), \\ 
B_3^{(1)}&=\frac{1}{4I_{11}^2I_{22}^3}(I_{13}u_k+I_{23}v_k)\left( 3I_{13}I_{22}^2I_{23}v^3_k-9I_{11}I_{13}I_{22}I_{23}u^2_kv_k
\right.\\
&\left. +I_{22}(2I_{22}I_{11}^2-2I_{22}^2I_{11}+4I_{22}I_{13}^2-5I_{23}^2I_{11})v^2_ku_k
-I_{11}(2I_{22}I_{13}^2-I_{23}^2I_{11})u^3_k\right).
\end{split}
\end{equation*}

It can be shown that the momentum locus $\mathfrak{u}^{(1)}_\varepsilon\subset \so(3)^*$ has pinch points and self-intersections. This 
is suggested by its graph in Figure \ref{F:MomLocusMosVes} and agrees with the discussion in Fedorov and Zenkov \cite{FeZen}. Away from these points,
the inverse of the discrete Legendre transformation is defined and is smooth, so we can locally define the discrete
time momentum mapping on $\mathfrak{u}^{(1)}_\varepsilon$ by
\begin{equation}
\label{E:DefBstarMosVes}
\mathcal{B}_{(1,\varepsilon)}^*=\mathbb{F}\ell_d^{(1,\varepsilon)} \circ \mathcal{B}_{(1,\varepsilon)} \circ (\mathbb{F}\ell^{(1,\varepsilon)}_d)^{-1}.
\end{equation} 

We will  compute the order of local truncation error of the approximation of the continuous flow of \eqref{E:EulerPoincareDual}  by
the projection of  $\mathcal{B}_{(1,\varepsilon)}^*$ onto the $M_1$-$M_2$ plane.

By the implicit function theorem, the equations \eqref{MParamea}, \eqref{MParameb}, can be inverted in  a vicinity of $(M_1,M_2)=(0,0)$ where
we consider $\varepsilon$ as a parameter. The   asymptotic expansion for the inversion as $\varepsilon \to 0$ is 
computed to be given by
\begin{equation}
\label{E:InvLegMosVes}
\begin{split}
u&=\frac{M_1}{I_{11}}+\varepsilon F_1 +\varepsilon^2 F_2^{(1)} + \varepsilon^3 F_3^{(1)} +  \mathcal{O}(\varepsilon^4), \\
v&=\frac{M_2}{I_{22}}+\varepsilon G_1 +\varepsilon^2 G_2^{(1)} + \varepsilon^3 G_3^{(1)} +  \mathcal{O}(\varepsilon^4),
\end{split}
\end{equation}
where
\begin{equation}
\label{E:InvLegAuxGeneral}
\begin{split}
F_1&=-\frac{1}{2I_{11}^2I_{22}^2}(I_{13}I_{22}M_1+I_{11}I_{23}M_2) M_2, \\
G_1&=\,\,\,\,\,\frac{1}{2I_{11}^2I_{22}^2}(I_{13}I_{22}M_1+I_{11}I_{23}M_2)M_1, 
\end{split}
\end{equation}
and
\begin{equation*}
\begin{split}
&F_2^{(1)}=\frac{1}{4I_{11}^4I_{22}^3}\left ( I_{22}^2(I_{11}I_{22}-I_{13}^2 )  M_1^3  + I_{11}(I_{11}^2I_{22}-2I_{11}
I_{23}^2+I_{13}^2I_{22})M_1M_2^2 \right . \\
  & \qquad \qquad \qquad \left .  -3I_{11}I_{13}I_{22}I_{23}M_1^2M_2 + I_{11}^2I_{13}I_{23}M_2^3\right ),  \\
&F_3^{(1)} = \frac{-1}{8I_{11}^5I_{22}^5}(I_{13}I_{22}M_1+I_{11}I_{23}M_2)\left( 2I_{13}I_{22}^2I_{23}M_1^3+I_{22}(4I_{11}I_{22}^2-2I_{11}^2I_{22} \right.\\
&\left. \qquad \qquad \qquad +3I_{11}I_{23}^2-4I_{13}^2I_{22})M_1^2M_2-8I_{11}I_{13}I_{22}I_{23}M_1M_2^2
\right.\\
&\left.  \qquad \qquad \qquad
+I_{11}(2I_{11}^2I_{22}-2I_{11}I_{23}^2+I_{13}^2I_{22})M_2^3\right), \end{split}
\end{equation*}\begin{equation*}\begin{split}
&G_2^{(1)}= \frac{1}{4I_{11}^3I_{22}^4}\left (I_{11}^2(I_{11}I_{22}-I_{23}^2 )M_2^3 + I_{22}(I_{11}I_{22}^2-2I_{22}I_{13}^2+I_{11}
I_{23}^2)M_1^2M_2  \right . \\  & \qquad \qquad \qquad  \left .  -3I_{11}I_{13}I_{22}I_{23}M_1M_2^2 
 +  I_{22}^2I_{23}I_{13}  M_1^3\right ), \\
&G_3^{(1)} = \frac{1}{8I_{11}^5I_{22}^5}(I_{13}I_{22}M_1+I_{11}I_{23}M_2)\left(2I_{23}I_{11}^2I_{13}M_2^3+I_{11}(4I_{22}I_{11}^2-2I_{22}^2I_{11}\right.\\
&\left. \qquad \qquad \qquad+3I_{22}I_{13}^2-4I_{23}^2I_{11})M_1M_2^2-8I_{11}I_{13}I_{22}I_{23}M_1^2M_2\right.\\
&\left. \qquad \qquad \qquad+ I_{22}(2I_{22}^2I_{11}-2I_{22}I_{13}^2+I_{23}^2I_{11})M_1^3\right).
\end{split}
\end{equation*}

On the other hand, equations \eqref{MParamea}, \eqref{MParameb} admit the asymptotic
expansion in $\varepsilon$
\begin{equation}
\begin{split}
\label{E:SeriesMuvMosVes}
M_1&=I_{11}u+\frac{\varepsilon}{2}(I_{13}u+I_{23}v)v-\frac{\varepsilon^2}{4}I_{11}u(u^2+v^2)\\ 
& \qquad \qquad  \qquad -\frac{\varepsilon^3}{8}v(I_{13}u+I_{23}v)(u^2+v^2)+\mathcal{O}(\varepsilon^4) , \\
M_2&=I_{22}v-\frac{\varepsilon}{2}(I_{13}u+I_{23}v)u-\frac{\varepsilon^2}{4}I_{22}v(u^2+v^2)\\ 
& \qquad \qquad  \qquad +\frac{\varepsilon^3}{8}u(I_{13}u+I_{23}v)(u^2+v^2) +\mathcal{O}(\varepsilon^4). 
\end{split}
\end{equation}

An asymptotic expansion for the projection of $\mathcal{B}_{(1,\varepsilon)}^*$  onto the
$M_1$-$M_2$ plane can now be obtained using 
\eqref{E:AsymMosVesB}, \eqref{E:InvLegMosVes} and  \eqref{E:SeriesMuvMosVes}. Denoting $((M_1)_{k+1}, (M_2)_{k+1})=(\tilde M_1, \tilde M_2)$
and $((M_1)_{k}, (M_2)_{k})=(M_1,M_2)$, we have
\begin{equation}
\label{E:ExpansionMMosVes}
\begin{split}
\tilde M_1&=M_1 + \varepsilon \mu_1 + \varepsilon^2 \mu_2 +  \varepsilon^3 \mu_3^{(1)}+\mathcal{O}(\varepsilon^4),  \\
\tilde M_2&=M_2 + \varepsilon \nu_1 + \varepsilon^2 \nu_2 +  \varepsilon^3 \nu_3^{(1)}+\mathcal{O}(\varepsilon^4),
\end{split}
\end{equation}
where $\mu_1, \mu_2, \nu_1, \nu_2,$ are given by \eqref{E:ExpansionMAuxGeneral}, and
\begin{equation}
\label{E:ExpansionMAuxMosVes}
\begin{split}
& \mu_3^{(1)}=\frac{-1}{4I_{11}^4I_{22}^5}\left ( I_{13}I_{22}M_1+I_{11}I_{23}M_2\right ) \left(  2 I_{13}I_{22}^2I_{23}M_1^3 -8I_{11}I_{13}I_{22}I_{23} M_1M_2^2 \right. \\
 &\left. +\left (  I_{11}I_{22}^3 +3I_{11}I_{23}^2I_{22}-4I_{13}^2I_{22}^2 \right ) M_1^2M_2 + \left ( I_{11}^3 I_{22} - 2I_{11}^2I_{23}^2 +I_{11}I_{22}I_{13}^2 \right ) M_2^3 \right), \\
& \nu_3^{(1)}=\frac{1}{4I_{11}^5I_{22}^4}\left ( I_{13}I_{22}M_1+I_{11}I_{23}M_2\right )  \left(  2 I_{13}I_{11}^2I_{23}M_2^3 -8I_{11}I_{13}I_{22}I_{23} M_1^2M_2 \right. \\
 &\left. + \left (  I_{11}^3I_{22} +3I_{11}I_{22}I_{13}^2-4I_{23}^2I_{11}^2 \right )  M_1M_2^2 + \left ( I_{11} I_{22}^3 - 2I_{22}^2I_{13}^2 +I_{22}I_{11}I_{23}^2 \right ) M_1^3  \right).
\end{split}
\end{equation}
Expansions \eqref{E:ExpansionMMosVes} are valid in a vicinity of $(M_1,M_2)=(0,0)$ and for a small $\varepsilon>0$.
Comparing \eqref{E:ExpansionMcont} and \eqref{E:ExpansionMAuxCont}, with \eqref{E:ExpansionMMosVes} and \eqref{E:ExpansionMAuxMosVes},
we conclude the following.
\begin{theorem}
\label{T:AccuracyMosVes}
The local truncation error of the approximation of the flow of the Euler-Poincar\'e-Suslov equations
 \eqref{E:EulerPoincareDual} by the  projection of the discrete time momentum mapping $\mathcal{B}_{(1,\varepsilon)}^*$  onto the 
$M_1$-$M_2$ plane is  second order.
\end{theorem}

\section{Analysis of the consistent discretisation $(\Ese,\ell^{(\infty,\varepsilon)}_d)$}
\label{S:reqInfty}

We repeat the analysis performed in \S\ref{S:req1} for the discretisation defined by $\ell^{(\infty,\varepsilon)}_d$ and $\Ese$. As in the previous
section, we denote $(\omega_1,\omega_2)=(u,v)$ to avoid double subscripts.

\subsection{The momentum locus $\mathfrak{u}_{\varepsilon}^{(\infty)}$} Upon the same considerations of \S\ref{E:MomLocusMosVes},
the discrete Legendre transformation \eqref{DiscMomentum}
associated to the discrete Lagrangian $\ell^{(\infty,\varepsilon)}_d$ given by \eqref{E:DiscLagCayley} is computed to be
\begin{equation}
\label{E:DiscLegTrans2}
\begin{split}
\mathbb{F}\ell^{(\infty,\varepsilon)}_d(W)&=\frac{2}{\varepsilon}( Wg(W)\J-g(W)\J W^T + Wg(W)\J h(W)g(W)  \\ 
\ &-g(W)\J h(W) g(W) W^T  \J g(W)W^T+W\J g(W) \\
&-g(W)h(W)\J g(W) W^T + Wg(W)h(W)\J g(W) ),
\end{split}
\end{equation}
where
\begin{equation*}
g(W)=(2+W+W^T)^{-1}, \qquad h(W)=2-W-W^T.
\end{equation*}

 Using 
 \eqref{E:DiscLegTrans2} and \eqref{E:Cayd}, we obtain the following expression for the components of 
 $M=\mathbb{F}\ell^{(\infty,\varepsilon)}_d(W)$ where $W=\mbox{Cay}_\varepsilon(u,v,0)$
\begin{subequations}\label{MParameInfty}
\begin{align}
M_1=& I_{11}u +\frac{\varepsilon}{2}v(I_{13}u+I_{23}v) 
+ \frac{\varepsilon^2}{4}u(I_{11}u^2+I_{22}v^2),\label{MParameInftya}\\
M_2=&I_{22}v -\frac{\varepsilon}{2}u(I_{13}u+I_{23}v) 
+ \frac{\varepsilon^2}{4}v(I_{11}u^2+I_{22}v^2),\label{MParameInftyb}\\
M_3= &I_{13}u+I_{23}v +\frac{\varepsilon}{2}uv(I_{22}-I_{11}).
\label{MParameInftyc}
\end{align}
\end{subequations}

Also in analogy with the discussion presented in  \S\ref{S:req1}, one can check that  the momentum locus $\mathfrak{u}^{(\infty)}_\varepsilon$ contains the origin, 
 where it is tangent to the  plane $\mathfrak{d}^*$ defined in \eqref{E:defdestarSuslov}. 
Figure \ref{F:MomLocusCayley} illustrates both the momentum locus $\mathfrak{u}^{(\infty)}_{\varepsilon}$ and the plane $\de^*$ for generic 
numerical values of $\I$ and a small value of $\varepsilon$.

\begin{figure}[h!]
\centering
\includegraphics[width=4.6cm]{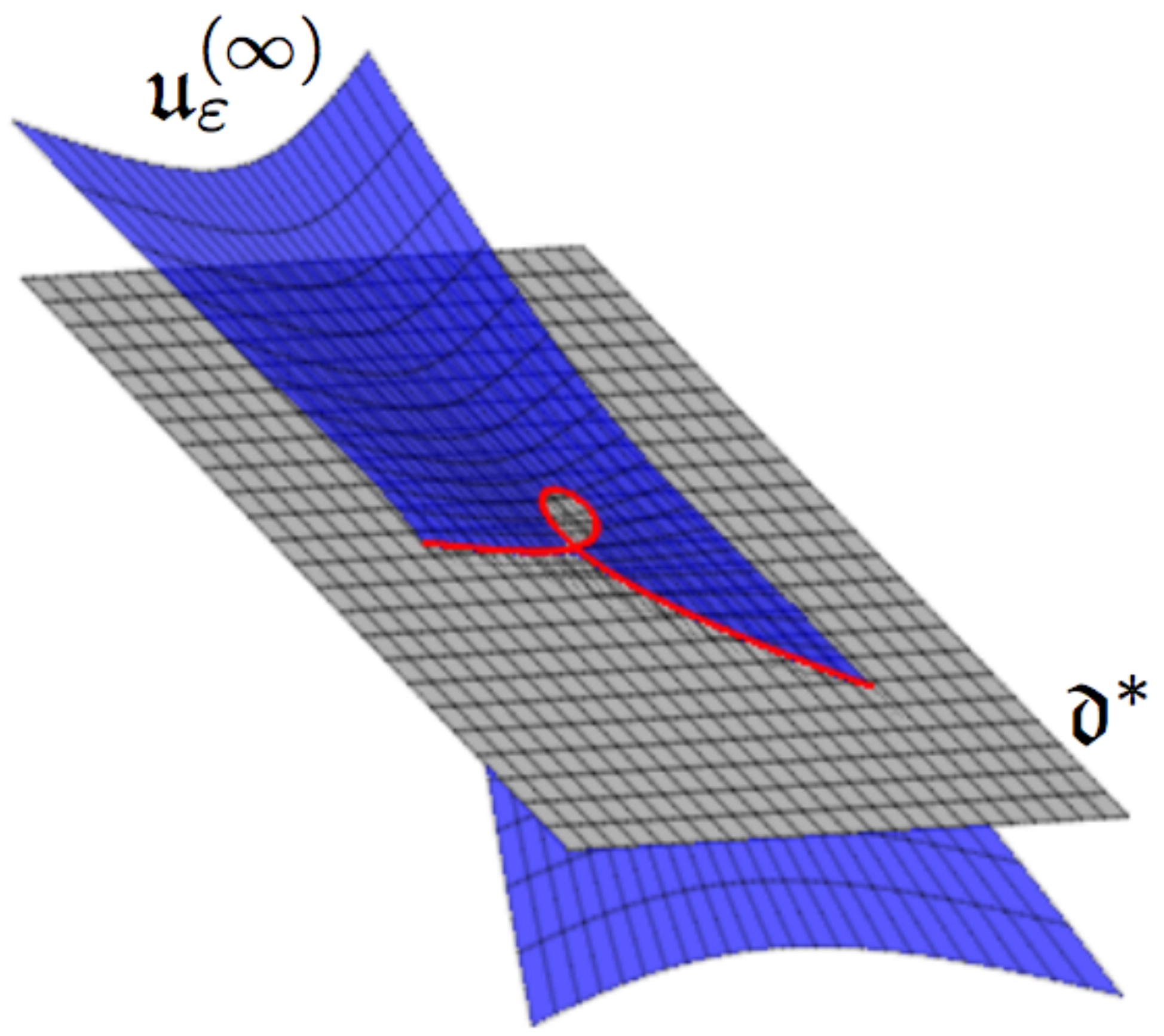}
\caption{{\small The discrete momentum locus $\mathfrak{u}^{(\infty)}_{\varepsilon}$ defined by $\ell_d^{(\infty,\varepsilon)}$ and the plane $\de^*$ immersed in $\so(3)^*=\R^3$. 
They intersect along the red curve, which self-intersects  at the origin where $\mathfrak{u}^{(\infty)}_{\varepsilon}$ is tangent to $\de^*$.
}}
\label{F:MomLocusCayley}
\end{figure}

Notice that $\mathfrak{u}^{(\infty)}_{\varepsilon}$ appears to be unbounded. This can be directly proved using the parametrisation  \eqref{MParameInfty}
and contrasts with the situation encountered in the study of the 
momentum locus defined by $\ell^{(1,\varepsilon)}_d$ and illustrated in Figure \ref{F:MomLocusMosVes}. In the Appendix we show that $\mathfrak{u}^{(\infty)}_{\varepsilon}$
is contained in the zero locus of a polynomial of degree $7$ in $(M_1,M_2,M_3)$.

\subsection{Discrete Euler-Poincar\'e-Suslov equations}

Using \eqref{E:Cayd}, \eqref{MParameInftya} and  \eqref{MParameInftyb}, the components of the Euler-Poincar\'e-Suslov equations \eqref{MdiscGene} 
that do not involve the multiplier $\lambda_k$ may be rewritten as a set of implicit equations to determine $(u_{k+1},v_{k+1})$ in terms of 
$(u_{k},v_{k})$. These equations are
\begin{equation}
\label{E:Alt-scheme-full}
\begin{split}
p_3^+(u_{k+1},v_{k+1})
 &=p_3^-(u_k,v_k) \, ,\\
q_3^-(u_{k+1},v_{k+1}) &=q_3^+(u_k,v_k)\, ,
\end{split}
\end{equation}
where $p_3^\pm$ and $q_3^\pm$ are the degree three polynomials 
\begin{equation*}
\begin{split}
p_3^\pm(u,v)&=I_{11} u\pm \frac{\varepsilon}{2}v \left (I_{13}u + I_{23}v \right ) + \frac{\varepsilon^2}{4}u\left ( I_{11}u^2 + I_{22}v^2 \right ), \\
q_3^\pm(u,v)&=I_{22} v \pm  \frac{\varepsilon}{2}u\left ( I_{13}u + I_{23}v\right )+ \frac{\varepsilon^2}{4}v\left ( I_{11}u^2 + I_{22}v^2 \right ).
\end{split}
\end{equation*}

In order to understand the number of solutions that \eqref{E:Alt-scheme-full} has, we compute the resultant of the polynomials $p_3^+-c_1, q_3^--c_2$ with respect
to the variable $u$, where $c_1$ and $c_2$ are scalars. This gives a polynomial of degree seven in $v$, say $\alpha(v)=\sum_{j=0}^7a_jv^j$. The coefficients $a_7$ and $a_6$  of this polynomial are 
\begin{equation*}
\begin{split}
a_7&=\frac{\varepsilon^8}{256}\left (I_{11}-I_{22})^2(I_{11}I_{23}^2+I_{22}I_{23}^2 \right ), \\
a_6&=\frac{\varepsilon^7}{64}\left (I_{11}-I_{22})(I_{11}^2I_{22}- I_{11}I_{22}^2-I_{11} I_{23}^2-I_{13}^2I_{22} \right ).
\end{split}
\end{equation*}
If the inertia tensor is non-diagonal and $I_{11}\neq I_{22}$, then $a_7\neq 0$ and  hence \eqref{E:Alt-scheme-full} admits exactly 7 (possibly complex) 
 solutions. This coincides with the degree of the polynomial $q_7$ given in Proposition \ref{Propoq7} of the Appendix. Experiments show that if $u_{k},v_{k}$
are real, then 1 solution is real and six are complex.

On the other hand, if $I_{11}=I_{22}$ then $\alpha$ has degree 5. In this case, the leading coefficient is
\begin{equation*}
a_5=\frac{I_{11}\varepsilon^6}{64}\left (I_{13}^2+I_{23}^2 \right ),
\end{equation*}
that does not vanish if the inertia tensor is non-diagonal. Hence, in this case \eqref{E:Alt-scheme-full} admits exactly 5 (possibly complex) solutions. Experiments show that if $u_{k},v_{k}$
are real, then 1 solution is real and four are complex.

\subsection{Approximation of the continuous flow}
\label{S:AsympExpSuslovCayley}

In analogy with the discussion of  \S\ref{S:AsympExpSuslovMoserVes}, one of the branches of  
the scheme \eqref{E:Alt-scheme-full} defines a single-valued discrete time map $\mathcal{B}_{(\infty,\varepsilon)}$
on $\Ese$ that approximates the flow of the continuous Suslov problem on $SO(3)$.  The asymptotic expansion for such  $(u_{k+1},v_{k+1})$ in terms of $(u_k,v_k)$ as
$\varepsilon\to 0$ is given by
\begin{equation}
\label{E:AsymCayleyB}
\begin{split}
u_{k+1}&=u_k + \varepsilon A_1 +\varepsilon^2 A_2+\varepsilon^3A_3^{(\infty)}  +\mathcal{O}(\varepsilon^4), \\
v_{k+1}&=v_k  + \varepsilon B_1 +\varepsilon^2 B_2+\varepsilon^3B_3^{(\infty)}  +\mathcal{O}(\varepsilon^4),
\end{split}
\end{equation}
where $A_1, A_2, B_1, B_2$ are defined in  
\eqref{E:AsymAuxGeneral} and
 \begin{equation*}
\label{E:AsymAuxCayley}
\begin{split}
A_3^{(\infty)}&= -\frac{1}{4I_{11}^3I_{22}^2}(I_{13}u_k+I_{23}v_k)\left(3I_{13}I_{11}^2I_{23}u^3_k-9I_{11}I_{13}I_{22}I_{23}u_kv^2_k\right.\\
&\left.\quad\quad\quad\quad\quad\qquad\qquad\qquad +I_{11}(4I_{11}I_{23}^2-I_{11}I_{22}^2-5I_{13}^2I_{22})u^2_kv_k\right.\\
&\left.\quad\quad\quad\quad\quad\qquad\qquad\qquad +I_{22}(I_{22}I_{13}^2-I_{11}I_{22}^2-2I_{11}I_{23}^2)v^3_k\right), \\
B_3^{(\infty)}&=\frac{1}{4I_{11}^2I_{22}^3}(I_{13}u_k+I_{23}v_k)\left(3I_{23}I_{22}^2I_{13}v_k^3-9I_{11}I_{13}I_{22}I_{23}v_ku^2_k\right.\\
&\left.\quad\quad\quad\quad\quad\qquad\qquad\qquad +I_{22}(4I_{22}I_{13}^2-I_{22}I_{11}^2-5I_{23}^2I_{11})v_k^2u_k\right.\\
&\left.\quad\quad\quad\quad\quad\qquad\qquad\qquad +I_{11}(I_{11}I_{23}^2-I_{22}I_{11}^2-2I_{22}I_{13}^2)u_k^3\right). 
\end{split}
\end{equation*}

Analogous to \eqref{E:DefBstarMosVes}, we can  locally define the discrete
time momentum mapping on $\mathfrak{u}^{(\infty)}_\varepsilon$ by
\begin{equation*}
\label{E:DefBstarCayley}
\mathcal{B}_{(\infty,\varepsilon)}^*=\mathbb{F}\ell_d^{(\infty,\varepsilon)} \circ \mathcal{B}_{(\infty,\varepsilon)} \circ (\mathbb{F}\ell^{(\infty,\varepsilon)}_d)^{-1}.
\end{equation*} 
We will now  compute the order of local truncation error of the approximation of the continuous flow of \eqref{E:EulerPoincareDual}  by
the projection of  $\mathcal{B}_{(\infty ,\varepsilon)}^*$ onto the $M_1$-$M_2$ plane as we did with $\mathcal{B}_{(1 ,\varepsilon)}^*$ in \S\ref{S:AsympExpSuslovMoserVes}.

The inversion of equations \eqref{MParameInftya} and \eqref{MParameInftyb} in a neighbourhood  of $(M_1,M_2)=(0,0)$ yields
\begin{equation}
\label{E:InvLegCayley}
\begin{split}
u&=\frac{M_1}{I_{11}}+\varepsilon F_1 +\varepsilon^2 F_2^{(\infty)} + \varepsilon^3 F_3^{(\infty)} +  \mathcal{O}(\varepsilon^4), \\
v&=\frac{M_2}{I_{22}}+\varepsilon G_1 +\varepsilon^2 G_2^{(\infty)} + \varepsilon^3 G_3^{(\infty)} +  \mathcal{O}(\varepsilon^4),
\end{split}
\end{equation}
where $F_1$, $G_1$ are given by \eqref{E:InvLegAuxGeneral},
and
\begin{equation*}
\label{E:InvLegAuxMosVes}
\begin{split}
F_2^{(\infty)}&=-\frac{1}{2I_{11}^3I_{22}^3}(I_{13}I_{22}M_1+I_{11}I_{23}M_2)(I_{13}I_{22}M_1^2+2I_{11}I_{23}M_1M_2-I_{11}I_{13}M_2^2),  \\
F_3^{(\infty)}& = \frac{-1}{4I_{11}^4I_{22}^5}\left(2I_{13}^2I_{22}^3I_{23}M_1^4
-I_{11}I_{22}I_{23}(2I_{11}^2I_{22}-3I_{11}I_{23}^2+12I_{13}^2I_{22})M_1^2M_2^2\right.\\
&\left.\qquad\qquad\quad\quad -I_{13}I_{22}(I_{11}^2I_{22}^2+I_{11}I_{22}^3-5I_{11}I_{22}I_{23}^2+4I_{13}^2I_{22}^2)M_1^3M_2 \right.\\
&\left.\qquad\qquad\quad\quad-I_{11}I_{13}I_{22}(I_{11}^2I_{22}+I_{11}I_{22}^2+10I_{11}I_{23}^2-I_{13}^2I_{22})M_1M_2^3\right.\\
&\left.\qquad\qquad\quad\quad-I_{11}I_{23}(2I_{11}^3I_{22}+2I_{11}^2I_{23}^2-I_{11}I_{13}^2I_{22})M_2^4\right), \\
G_2^{(\infty)}&= \frac{-1}{2I_{11}^3I_{22}^3}(I_{13}I_{22}M_1+I_{11}I_{23}M_2)(I_{23}I_{11}M_2^2+2I_{22}I_{13}M_1M_2-I_{22}I_{23}M_2^2),  \\
G_3^{(\infty)}& =\frac{1}{4I_{11}^5I_{22}^4}\left(2I_{23}^2I_{11}^3I_{13}M_2^4 -I_{22}I_{11}I_{13}(2I_{22}^2I_{11}-3I_{22}I_{13}^2+12I_{23}^2I_{11})M_2^2M_1^2\right.\\
&\left.\qquad\qquad\quad\quad -I_{23}I_{11}(I_{22}^2I_{11}^2+I_{22}I_{11}^3-5I_{22}I_{11}I_{13}^2+4I_{23}^2I_{11}^2)M_2^3M_1 \right.\\
&\left.\qquad\qquad\quad\quad-I_{22}I_{23}I_{11}(I_{22}^2I_{11}+I_{22}I_{11}^2+10I_{22}I_{13}^2-I_{23}^2I_{11})M_2M_1^3\right.\\
&\left.\qquad\qquad\quad\quad-I_{22}I_{13}(2I_{22}^3I_{11}+2I_{22}^2I_{13}^2-I_{22}I_{23}^2I_{11})M_1^4\right).
\end{split}
\end{equation*}

An asymptotic expansion for the projection of  $\mathcal{B}_{(\infty,\varepsilon)}^*$ onto the plane
$M_1\, M_2$ is obtained using 
\eqref{E:AsymCayleyB}, \eqref{E:InvLegCayley} and  \eqref{MParameInfty}. Denoting $((M_1)_{k+1}, (M_2)_{k+1})=(\tilde M_1, \tilde M_2)$
and $((M_1)_{k}, (M_2)_{k})=(M_1,M_2)$, we have
\begin{equation}
\label{E:ExpansionMCayley}
\begin{split}
\tilde M_1&=M_1 + \varepsilon \mu_1 + \varepsilon^2 \mu_2 +  \varepsilon^3 \mu_3^{(\infty)}+\mathcal{O}(\varepsilon^4),  \\
\tilde M_2&=M_2 + \varepsilon \nu_1 + \varepsilon^2 \nu_2 +  \varepsilon^3 \nu_3^{(\infty)}+\mathcal{O}(\varepsilon^4),
\end{split}
\end{equation}
where $\mu_1, \mu_2, \nu_1, \nu_2,$ are given by \eqref{E:ExpansionMAuxGeneral}, and
\begin{equation}
\label{E:ExpansionMAuxCayley}
\begin{split}
 &\mu_3^{(\infty)}=\frac{-1}{4 I_{11}^4I_{22}^5}\left(  2I_{13}^2I_{22}^3I_{23}M_1^4
 +I_{11}I_{22}I_{23}(3I_{11}I_{23}^2-2I_{11}^2I_{22}-12I_{13}^2I_{22})M_1^2M_2^2
 \right.\\
&\left. -I_{11}I_{22}(I_{11}I_{13}I_{22}^2+I_{13}I_{22}^3-5I_{13}I_{22}I_{23}^2+4I_{11}^2I_{22}^2)M_1^3M_2+I_{11}I_{13}I_{22}(I_{13}^2I_{22}\right.\\
&\left.-I_{11}^2I_{22} -I_{11}I_{22}^2-10I_{11}I_{23}^2)M_1M_2^3-I_{11}I_{23}(2I_{11}^3I_{22}+2I_{11}^2I_{23}^2-I_{11}I_{13}^2I_{22})M_2^4\right),  \\
 & \nu_3^{(\infty)}=\frac{1}{4 I_{11}^5I_{22}^4}\left(  2I_{23}^2I_{11}^3I_{13}M_2^4+I_{22}I_{11}I_{13}(3I_{22}I_{13}^2-2I_{22}^2I_{11}-12I_{23}^2I_{11})M_2^2M_1^2
\right.\\
&\left.  -I_{22}I_{11}(I_{22}I_{23}I_{11}^2+I_{23}I_{11}^3-5I_{23}I_{11}I_{13}^2+4I_{22}^2I_{11}^2)M_2^3M_1+I_{22}I_{23}I_{11}(I_{23}^2I_{11}\right.\\
&\left.-I_{22}^2I_{11} -I_{22}I_{11}^2-10I_{22}I_{13}^2)M_2M_1^3-I_{22}I_{13}(2I_{22}^3I_{11}+2I_{22}^2I_{13}^2-I_{22}I_{23}^2I_{11})M_1^4\right). 
\end{split}
\end{equation}
Expansions \eqref{E:ExpansionMCayley} are valid in a vicinity of $(M_1,M_2)=(0,0)$ and for a small $\varepsilon>0$.

Comparing \eqref{E:ExpansionMcont} and \eqref{E:ExpansionMAuxCont}, with \eqref{E:ExpansionMCayley} and \eqref{E:ExpansionMAuxCayley},
we conclude the following.
\begin{theorem}
\label{T:AccuracyCayley}
The local truncation error of the approximation of the flow of the Euler-Poincar\'e-Suslov equations
 \eqref{E:EulerPoincareDual} by the  projection of the discrete time momentum mapping $\mathcal{B}_{(\infty,\varepsilon)}^*$  onto the 
$M_1$-$M_2$ plane is  second order.
\end{theorem}

Theorems  \ref{T:AccuracyMosVes} and \ref{T:AccuracyCayley} show the local truncation error for both discretisations is of the same order. This indicates that consistency of the discretisation of a nonholonomic geometric integrator is not the essential feature
to guarantee an approximation of the continuous flow within  with a desired order of accuracy. This seems to contradict the Remark 3.1 in \cite{CoMa} quoted in the introduction.

\section{Discrete evolution of the energy } \label{S:Energy}

Here we discuss the energy preservation properties of both discretisations.

\subsection{Non-consistent  discretisation} \label{S:EnergyMosVes}

It is shown by Fedorov-Zenkov \cite{FeZen} that the  discretisation of the Suslov problem defined by $\Ese$ and $\ell_d^{(1,\varepsilon)}$ exactly preserves the restriction
of the constrained energy $E_c$ defined by \eqref{E:EnergySuslov} to the momentum
locus $\mathfrak{u}_\varepsilon^{(1)}$.
At the group level,  this conservation law  is formulated in the following.
\begin{proposition}\label{RationalFunction}
The discrete system \eqref{E:F-Z-scheme-full} evolves on the level sets of the rational function
\begin{equation*}
R(u,v)=\frac{4(I_{11}u^2+I_{22}v^2)(4I_{11}I_{22}+\varepsilon^2(I_{13}u+I_{23}v)^2)}{(4+\varepsilon^2(u^2+v^2))^2}.
\end{equation*}
\end{proposition}
\begin{proof}
A direct calculation shows that
\begin{equation*}
\label{E:RatFcnSuslov}
\begin{split}
 R(u_{k+1},v_{k+1})=
&\,\,\,\,\,\,I_{22}\left (\frac{4I_{11}u_{k+1}+2\varepsilon v_{k+1}(I_{13}u_{k+1}+I_{23}v_{k+1})}{4+\varepsilon^2(u_{k+1}^2+v_{k+1}^2)}  \right )^2\\
& +I_{11}\left (\frac{4I_{22}v_{k+1}-2\varepsilon u_{k+1}(I_{23}v_{k+1}+I_{13}u_{k+1})}{4+\varepsilon^2(u_{k+1}^2+v_{k+1}^2)}   \right )^2 .
\end{split}
\end{equation*}
So, using  \eqref{E:F-Z-scheme-full}, we get
\begin{equation*}
\begin{split}
R(u_{k+1},v_{k+1})&=I_{22}\left (\frac{4I_{11}u_{k}-2\varepsilon v_{k}(I_{13}u_{k}+I_{23}v_{k})}{4+\varepsilon^2(u_{k}^2+v_{k}^2)}  \right )^2\\
& +I_{11}\left (\frac{4I_{22}v_{k}+2\varepsilon u_{k}(I_{23}v_{k}+I_{13}u_{k})}{4+\varepsilon^2(u_{k}^2+v_{k}^2)}   \right )^2,
\end{split}
\end{equation*}
but the expression on the right hand side equals $R(u_{k},v_{k})$.
\end{proof}

It is important to notice that {\em all}  branches of the multi-valued discrete map defined by  \eqref{E:F-Z-scheme-full}  possess this
invariant.

The relationship between the rational function $R$ and  $E_c$ is established by noticing that, in view of
\eqref{MParamea}, \eqref{MParameb}, the relation $R(u_{k+1},v_{k+1})=R(u_{k},v_{k})$ can be rewritten as
\begin{equation*}
I_{22}(M_1)_{k+1}^2+I_{11} (M_2)_{k+1}^2 =I_{22}(M_1)_{k}^2+I_{11} (M_2)_{k}^2.
\end{equation*}
But, up to multiplication by a constant factor,  $I_{22}M_1^2+I_{11} M_2^2$ coincides with $E_c$.

The proposition above shows that the discrete map defined by  \eqref{E:F-Z-scheme-full} evolves on  
the algebraic curve
\begin{equation*}
\{(u,v)\in \mathbb{C}^2 : 4(I_{11}u^2+I_{22}v^2)(4I_{11}I_{22}+\varepsilon^2 (I_{13}u+I_{23}v)^2) -h(4+\varepsilon^2(u^2+v^2))^2=0\}
\end{equation*}
where $h=R(u_0,v_0)$.
The MAPLE package ``algcurves"  indicates that the compactification of $\mathcal{C}$ has genus 3 and is not hyperelliptic.

\subsection{Consistent discretisation}
\label{S:EnergyCayley}

The consistent discretisation defined by $\Ese$ and  $\ell_d^{(\infty, \varepsilon)}$ only preserves the 
constrained energy $E_c$ defined by \eqref{E:EnergySuslov} 
if $I_{11}=I_{22}$. One can show that   this condition implies that  the  
  axis of forbidden rotations of the body  $e_3$ is contained in an inertia eigen-plane.\footnote{Such condition also leads to simplifications in the properties of the continuous time 
  Suslov problem \cite{FeMaPr}, \cite{GNetal}.}

At the group level,   we have  the following.

\begin{proposition}\label{P:ConsEnergyCayley} Let $Q$ be the following polynomial in two variables
\begin{equation*}
\begin{split}
Q(u,v)&= \frac{1}{16}(I_{11}u^2+I_{22}v^2) \left ( 16 I_{11}I_{22} + \varepsilon^2 \left (8  I_{11}I_{22}(u^2+v^2) + 4(I_{13}u+I_{23}v)^2 \right )  \right . \\
& \left . \qquad  +4 \varepsilon^3  uv(I_{22}-I_{11})(I_{13}u+I_{23}v )   + \varepsilon^4(I_{11} u^2+I_{22}v^2)(I_{22}u^2+I_{11}v^2) \right ).
\end{split}
\end{equation*}
\begin{enumerate}
\item  If $I_{11}=I_{22}$ then 
the discrete system \eqref{E:Alt-scheme-full} evolves on the level sets of $Q$.
\item If $I_{11}\neq I_{22}$ then for successive points defined by   \eqref{E:Alt-scheme-full}, we have
\begin{equation*}
Q(u_{k+1},v_{k+1})-Q(u_{k},v_{k})=\mathcal{O}(\varepsilon^3).
\end{equation*}
\end{enumerate}
\end{proposition}
\begin{proof} 

A direct calculation shows that 
\begin{equation*}
I_{22}p_3^+(u_{k+1},v_{k+1})^2+I_{11}q_3^-(u_{k+1},v_{k+1})^2 = Q(u_{k+1},v_{k+1}),
\end{equation*}
and also
\begin{equation*}
\begin{split}
&I_{22} p_3^-(u_{k},v_{k})^2+I_{11}q_3^+(u_{k},v_{k})^2 = Q(u_{k},v_{k}) + \\  & \qquad \qquad \qquad \qquad\frac{\varepsilon^3u_kv_k}{2}(I_{11}-I_{22})(I_{11}u_k^2+I_{22}v_k^2)(I_{13}u_k+I_{23}v_k).
\end{split}
\end{equation*}
It follows from \eqref{E:Alt-scheme-full} that
\begin{equation*}
Q(u_{k+1},v_{k+1})=Q(u_{k},v_{k})+ \frac{\varepsilon^3u_kv_k}{2}(I_{11}-I_{22})(I_{11}u_k^2+I_{22}v_k^2)(I_{13}u_k+I_{23}v_k).
\end{equation*}

\end{proof}

Note that this proposition applies to all   branches of the multi-valued discrete map defined by  \eqref{E:Alt-scheme-full}.
In particular, if $I_{11}=I_{22}$, all 5 branches preserve $Q$. Given that in this case $e_3$ lies on an inertia eigen-plane, by performing 
a rotation of the body frame about $e_3$, we may assume
that $I_{13}$ vanishes and $Q$ takes the simplified form
\begin{equation*}
Q(u,v)=\frac{I_{11}}{16}(u^2+v^2)\left ( 16I_{11}^2 +\varepsilon^2( 8I_{11}^2(u^2+v^2) +4I_{23}^2v^2) + \varepsilon^4I_{11}^2(u^2+v^2)^2 \right ).
\end{equation*}
Hence, the multi-valued map \eqref{E:Alt-scheme-full} evolves on the algebraic curve 
\begin{equation*}
\mathcal{C}=\{(u,v)\in \mathbb{C}^2 \, : \, Q(u,v)-h=0\},
\end{equation*}
where the constant $h=Q(u_0,v_0)$. According to the MAPLE package ``algcurves" the compactification of $\mathcal{C}$ has genus 4 and is not hyperelliptic.

The relationship between the polynomial $Q$ and the constrained energy $ E_c$ given by  \eqref{E:EnergySuslov}  is established by noting that
in view of \eqref{MParameInftya} and \eqref{MParameInftyb} we have
\begin{equation*}
Q(u_{k},v_{k})=I_{22}(M_1)_{k}^2+I_{11} (M_2)_{k}^2,
\end{equation*}
which, up to multiplication by a constant factor,  coincides with $E_c$.

\begin{remark}
 If $I_{11}\neq I_{22}$, the proposition states that the approximation of $E_c$ by the discrete flow is  
$\mathcal{O}(\varepsilon^3)$ which is considerably good. One could imagine that in this case there still exists an energy-like 
first integral of the discrete system, for instance an $\mathcal{O}(\varepsilon)$-perturbation of $E_c$. We were unable to prove or disprove this possibility
that was brought to our attention by one of the anonymous referees of the paper.
\end{remark}

\begin{remark}
Note that in the energy-preserving cases mentioned above, the discrete flow at a fixed energy value defines a multi-valued map on an algebraic curve. Such curve
has symmetries and is a covering of another curve with lower genus. As  future work, it is interesting to  investigate if an exact analytical expression for the $n^{th}$ iterate
of these maps can be obtained in any of these curves.
\end{remark}

\section{Numerical simulations}\label{Simulations}

In this section we compare  the performance of the integrators $\mathcal{B}_{(1,\varepsilon)}^*$ and $\mathcal{B}_{(\infty,\varepsilon)}^*$
by conducting    numerical experiments. We consider two different choices of the inertia tensor. The first one is generic while the second one
satisfies the condition $I_{11}=I_{22}$ that ensures that $\mathcal{B}_{(\infty,\varepsilon)}^*$ exactly preserves $E_c$. In both cases we 
observe that the  approximation by the non-consistent integrator $\mathcal{B}_{(1,\varepsilon)}^*$ is slightly better.

\subsection{A generic inertia tensor}
Throughout this section we suppose that the inertia tensor has
\begin{equation*}
I_{11}=3, \qquad I_{22}=4, \qquad I_{33}=5, \qquad I_{13}=1, \qquad I_{23}=\frac{1}{2}.
\end{equation*}
We will approximate the solution $(M_1(t), M_2(t))$ to the Euler-Poincar\'e-Suslov equations \eqref{E:EulerPoincareDual} having initial condition
\begin{equation*}
\label{E:initcondGeneric}
M_1(0)=41.07400078, \qquad M_2(0)=-99.38251558,
\end{equation*}
over the time interval $0\leq t\leq 1$. Such solution can be obtained analytically, see e.g. \cite{FeMaPr}. This choice of initial condition is such that $(M_1(t),M_2(t))$ 
will evolve from one asymptotic state to another 
over the time interval $0\leq t \leq 1$. 

Our numerical experiments show that the momentum integrator $\mathcal{B}_{(1,\varepsilon)}^*$ corresponding to the non-compatible discretisation 
defined by $\Ese$ and  $\ell_d^{(1,\varepsilon)}$ approximates $(M_1(t), M_2(t))$ for $0\leq t\leq 1$  if the 
the time step $\varepsilon  \leq 0.3$. 
For $\varepsilon\geq 0.4$ we cannot generate sufficient iterations 
for the approximation. The problem arises since  the map
$(\F\ell_d^{(1,\varepsilon)})^{-1}$ that appears in   \eqref{E:DefBstarMosVes} only
exists locally.  We did not find this type of restriction on the time step for the momentum integrator $\mathcal{B}_{(\infty,\varepsilon)}^*$ corresponding to the 
consistent discretisation. In Figure   \ref{F:Generic} below we show approximations of $(M_1(t),M_2(t))$ by both integrators for $\varepsilon= 0.015, 0.030$. On the insets we show respectively
the  approximation of the constrained energy $E_c$  and the signed distance $\rho$   to the constraint subspace $\de^*$ given by 
\eqref{E:defdestarSuslov} at each time step. For the latter graph it is necessary to compute the iterates of all components of the vector $M$.

\begin{figure}[h!]
\centering
\includegraphics[width=15cm]{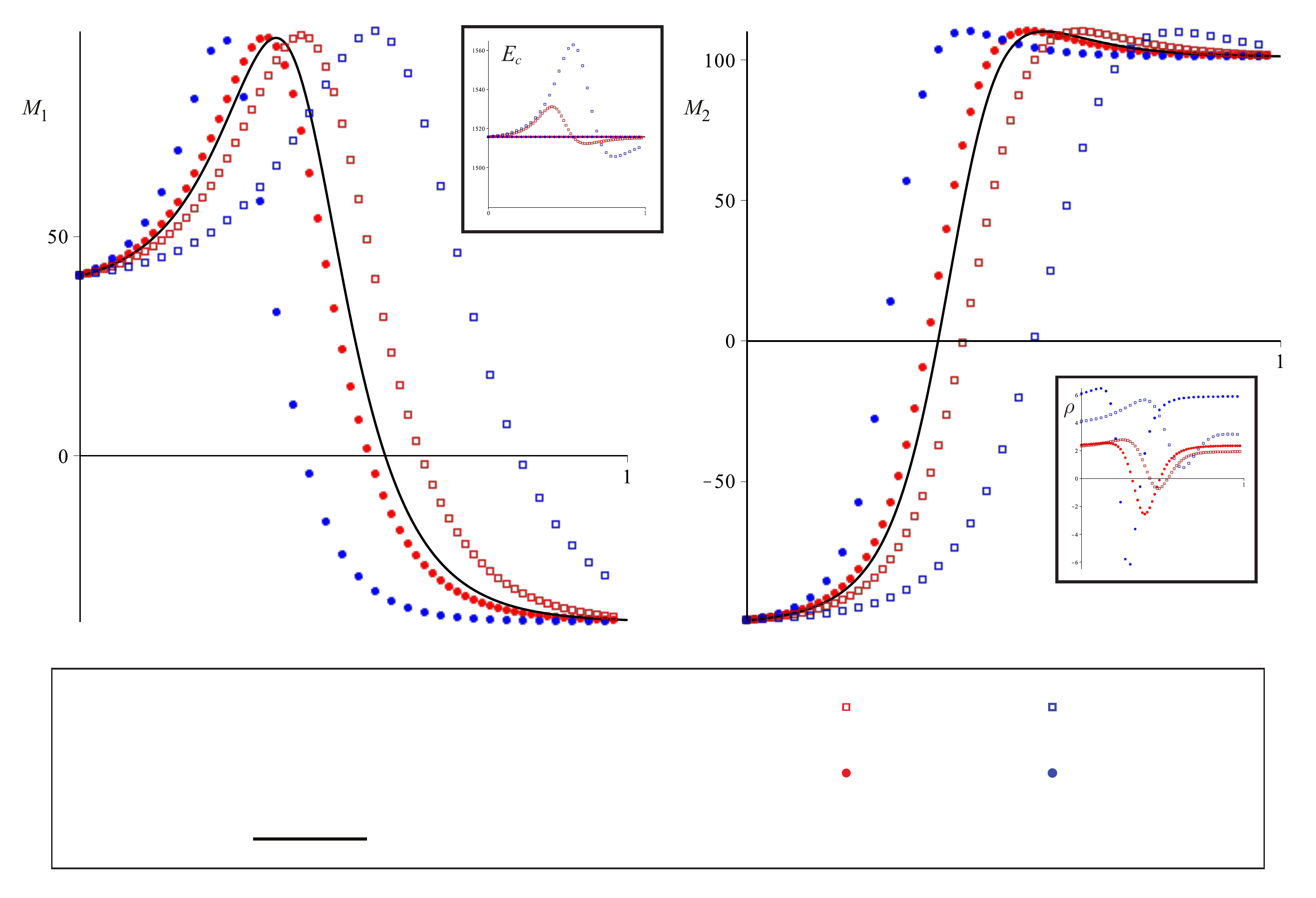}
 \put (-143,60) {$\varepsilon=0.015$} \put (-76,60) {$\varepsilon=0.030$} \put (-143,38) {$\varepsilon=0.015$}  \put (-76,38) {$\varepsilon=0.030$}
  \put (-400,60){Consistent approximation defined by $\ell_d^{(\infty,\varepsilon)} $}
    \put (-400,38){Non-consistent approximation defined by $\ell_d^{(1,\varepsilon)} $ }
     \put (-400,17){Exact}
\caption{{\small Comparison between  the  two discretisations  for a generic inertia tensor. The graph on the left shows the  approximation of $M_1$ and the evolution of the energy $E_c$ 
(inset). The graph on the right shows the  approximation of $M_2$ and the evolution of the signed distance $\rho$ to the constraint subspace $\de^*$ (inset).
}}
\label{F:Generic}
\end{figure}
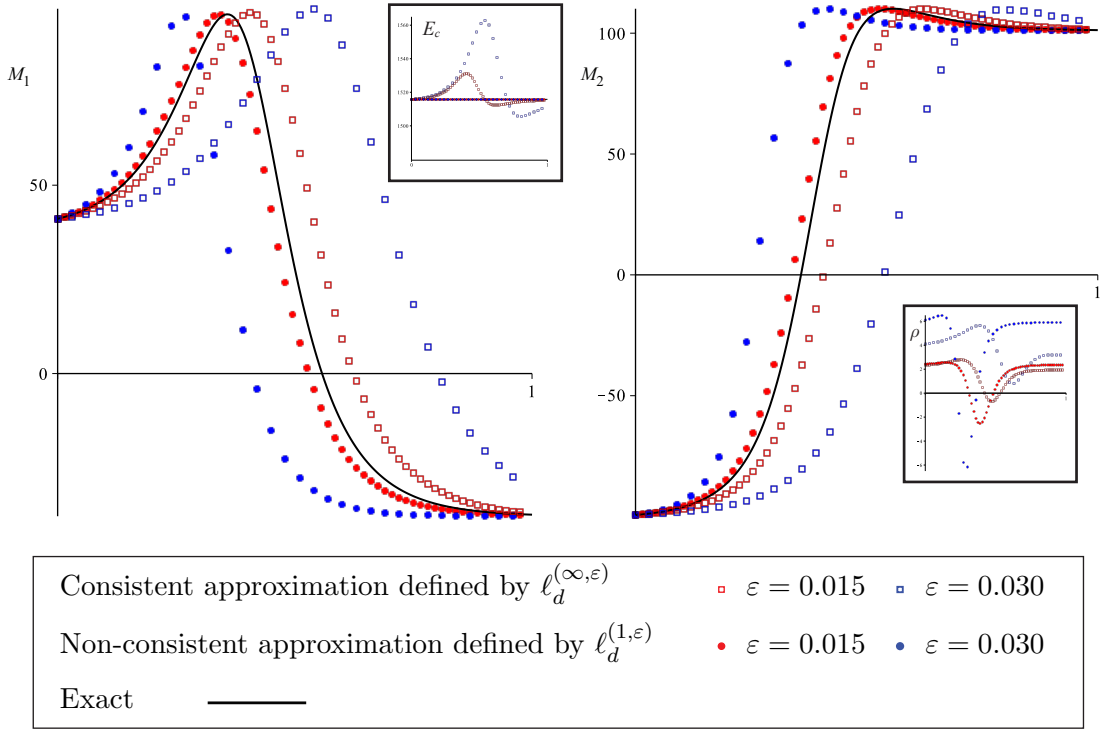

\subsection{A special inertia tensor}
Now we consider an inertia tensor having 
\begin{equation*}
I_{11}=3, \qquad I_{22}=3, \qquad I_{33}=5, \qquad I_{13}=0, \qquad I_{23}=\frac{1}{2},
\end{equation*}
and we approximate the solution to the Euler-Poincar\'e-Suslov equations \eqref{E:EulerPoincareDual} having initial condition
\begin{equation*}
\label{E:initcondSpecial}
M_1(0)=179.9836568, \qquad M_2(0)=2.4255507998,
\end{equation*}
Once again, the choice of initial condition is such that the solution will evolve from one asymptotic state to another 
over the time interval $0\leq t \leq 1$. 

This time the non-consistent  momentum integrator $\mathcal{B}_{(1,\varepsilon)}^*$ 
has problems if    $\varepsilon\geq 0.018$.
Figure \ref{F:Special} below is analogous to Figure \ref{F:Generic} in the previous section for $\varepsilon= 0.007, 0.014$.

\begin{figure}[h!]
\centering
\includegraphics[width=15cm]{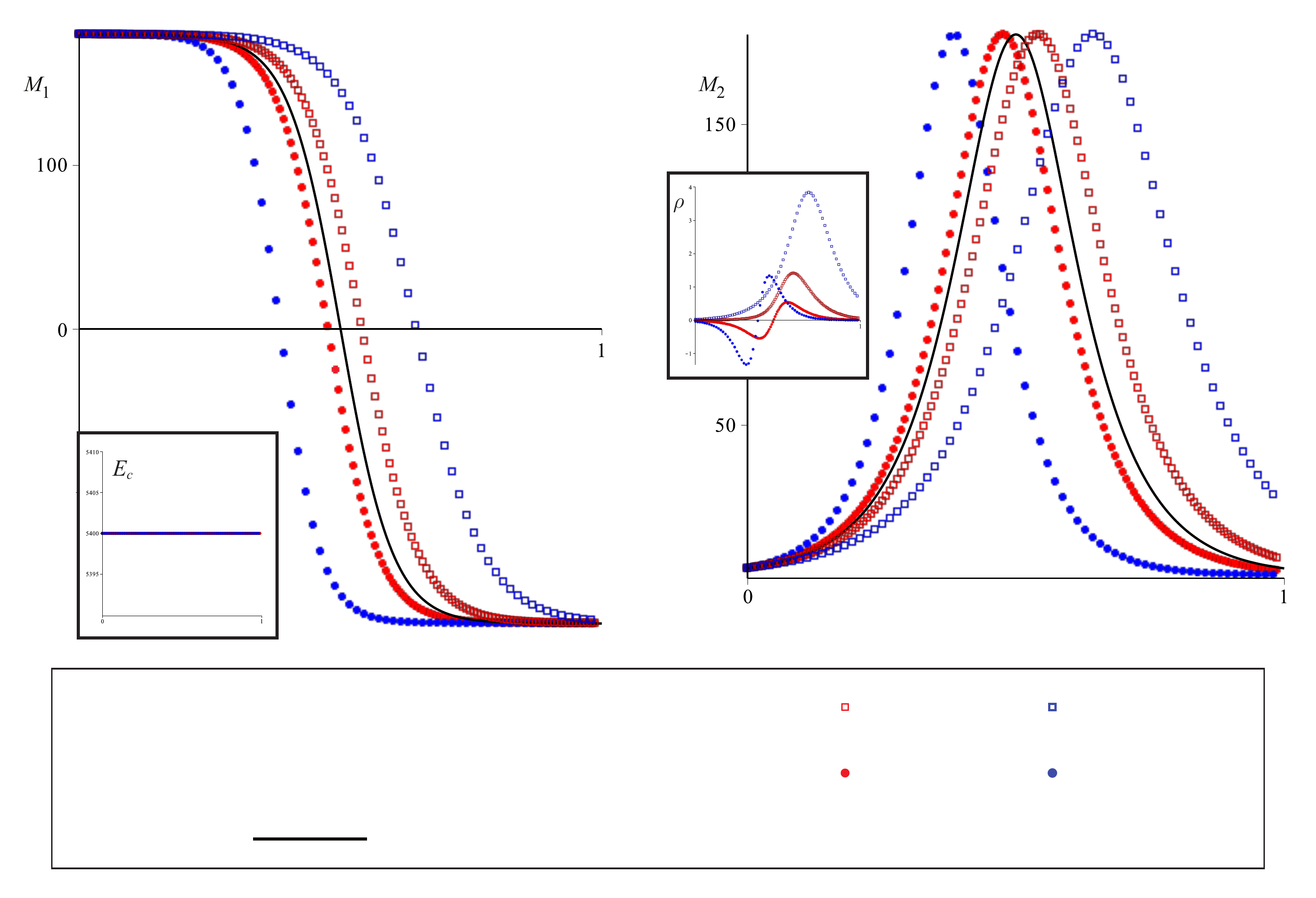}
 \put (-143,60) {$\varepsilon=0.007$} \put (-76,60) {$\varepsilon=0.014$} \put (-143,38) {$\varepsilon=0.007$}  \put (-76,38) {$\varepsilon=0.014$}
  \put (-400,60){Consistent approximation defined by $\ell_d^{(\infty,\varepsilon)} $}
    \put (-400,38){Non-consistent approximation defined by $\ell_d^{(1,\varepsilon)} $ }
     \put (-400,17){Exact}
\caption{{\small Comparison between of the  two discretisations  for a special inertia tensor having $I_{11}=I_{22}$. The graph on the left shows the  approximation of $M_1$ and the evolution of the energy $E_c$ 
(inset). The graph on the right shows the  approximation of $M_2$ and the evolution of the signed distance $\rho$ to the constraint subspace $\de^*$ (inset).
}}
\label{F:Special}
\end{figure}
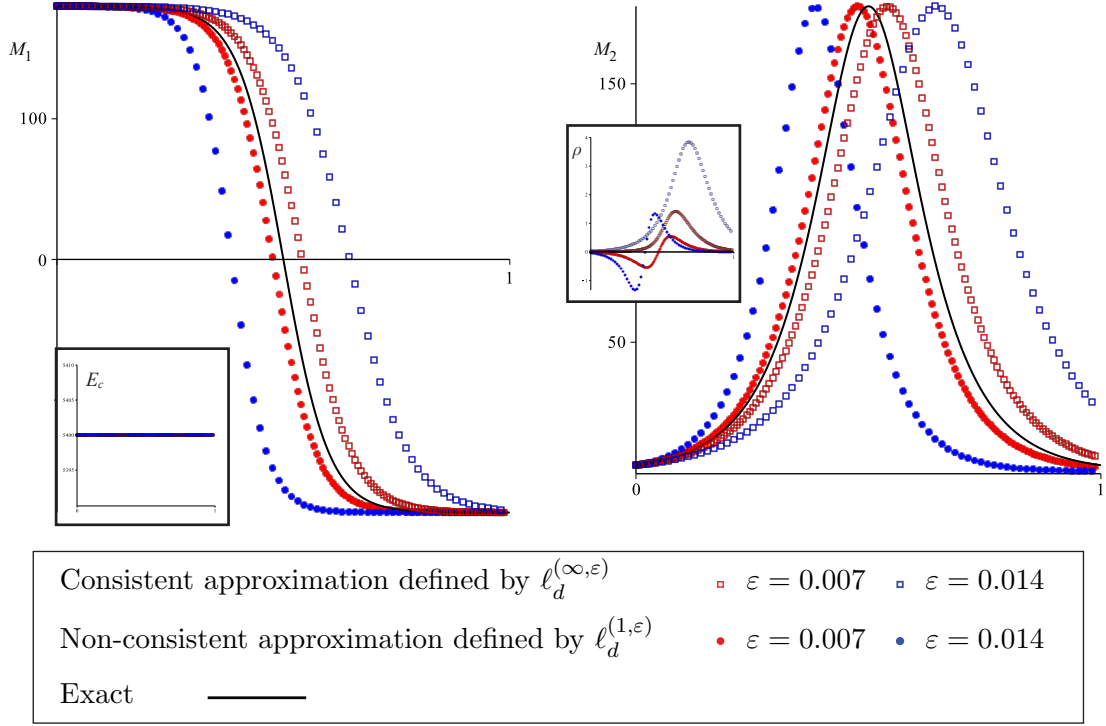

\begin{remark}
{\rm Note that the numerical experiments  show that the evolution of the non-consistent discrete map $\mathcal{B}_{(1,\varepsilon)}^*$ is slower than the continuous 
dynamics while the evolution of the consistent discrete map $\mathcal{B}_{(\infty,\varepsilon)}^*$ is faster. We have no explanation for this phenomenon.}
\end{remark}

\section{Conclusions}
\label{S:Conclusions}

To our knowledge, this work is the first in which the 
usefulness  of the consistency   condition of a discretisation of a nonholonomic system is explored.  
Our results indicate that    consistency may not
be the essential property to consider in order   to construct nonholonomic integrators. Indeed, for the specific discretisations
of the Suslov problem that we considered, we have shown that
\begin{enumerate}
\item Both the consistent and the
non-consistent discretisations  approximate the continuous flow of the system with local truncation errors of the same order.
\item The non-consistent discretisation  preserves the energy of the continuous system for an arbitrary inertia tensor. On the other hand, 
the consistent discretisation only preserves the energy of the continuous system for a  family of rigid bodies whose inertia tensor has a specific type of symmetry, and it is unclear
if an energy-like first integral exists in the general case.
\item Our numerical experiments indicate that the  non-consistent discretisation gives rise to an integrator that performs better than the consistent one for the same value of the time step.
\end{enumerate}

\section*{Acknowledgments} We are grateful to the two anonymous referees  for their detailed comments and remarks that helped us to improve
both the content and the presentation of the paper.

We acknowledge the support of the project B4 of
SFB-Transregio 109 ``Discretization in Geometry and Dynamics" and the hospitality of Technische Universit\"at Berlin during a joint visit of
the authors in  April 2015 when this project was started. LGN also benefited from the hospitality at TU Berlin during the months of June and November of 2015, and these visits were very valuable to make progress in this project.

We are  thankful to Yuri Fedorov, Yuri Suris and Claude Viallet  for very useful comments and remarks. Also to Matteo Petrera for 
his guidance with MAPLE and to Ramiro Ch\'avez-Tovar for his help in the edition of the figures in the paper.

LGN acknowledges support of the Program UNAM-DGAPA-PAPIIT IA103815 that served to cover travel expenses and to fund a visit of FJ to UNAM in Mexico in February 2016.

The research of FJ has been partially supported in its late stages by DGIST Research
and Development Program (CPS Global Center) funded by the Ministry of Science,
ICT \& Future Planning, Global Research Laboratory Program (2013K1A1A2A02078326)
through NRF, and Institute for Information \& Communications Technology Promotion (IITP) grant funded by the Korean government (MSIP) (No. B0101-15-0557,
Resilient Cyber-Physical Systems Research).

\section*{Appendix}

\begin{proposition}\label{4orderp}
The momentum locus $\mathfrak{u}^{(1)}_{\varepsilon}\subset \so(3)^*$ defined by  $\ell^{(1,\varepsilon)}_d$ is  contained in
the zero locus of the degree four polynomial $p_4(M_1,M_2,M_3;\varepsilon)$ given by
\begin{equation*}
\begin{split}
p_4(M_1,M_2,M_3;\varepsilon)=&-4(I_{11}-I_{22}) (I_{13}I_{22}M_1+I_{11}I_{23}M_2-I_{11}I_{22}M_3) \\
& \qquad +\sum_{j=2}^4 \varepsilon^{j-1} p_4^{(j)}(M_1,M_2,M_3),
\end{split}
\end{equation*}
where, for $j=2,3,4$, $p_4^{(j)}$ is a homogeneous polynomial of degree $j$. 
\end{proposition}

The proof of this proposition  follows by a direct substitution of \eqref{MParame} into $p_4$, whose expression  was obtained with the help of a symbolic mathematical software. The polynomials $p_4^{(j)}$ are
\begin{small}
\[
\begin{split}
p_4^{(2)}(X,Y,Z)& =
 -\frac{2I_{11}I_{22}}{I_{23}I_{13}}(I_{13}^2+I_{23}^2)Z^2+ \frac{2I_{22}I_{13}}{I_{23}}(I_{11}-I_{22})X^2+4(I_{11}-I_{22})^2XY\\
&+\frac{2}{I_{13}}(I_{11}^2I_{22}+2I_{11}I_{13}^2-I_{11}I_{22}^2+I_{11}I_{23}^2-I_{13}^2I_{22})YZ-\frac{I_{11}I_{23}}{I_{13}}(I_{11}-I_{22})Y^2\\
& - \frac{2}{I_{23}}(I_{11}^2I_{22}-I_{11}I_{22}^2+I_{11}I_{23}^2-I_{13}^2I_{22}-2I_{22}I_{23}^2)XZ ,
\end{split}
\]
\[
\begin{split}
p_4^{(4)}(X,Y,Z)&=\frac{I_{22}}{I_{23}}XY^2Z-XYZ^2-\frac{1}{2I_{23}I_{13}}(I_{13}^2+I_{23}^2)Z^4-X^3Y-XY^3-\frac{1}{2}X^4\\
&+\frac{I_{11}}{I_{13}}YZ^3+\frac{I_{11}}{I_{13}}Y^3Z+\frac{I_{11}}{I_{13}}X^2YZ+\frac{I_{22}}{I_{23}}XZ^3+\frac{I_{22}}{I_{23}}X^3Z-\frac{I_{23}}{2I_{13}}Y^4 \\
&-\frac{1}{2I_{13}I_{23}}(I_{11}^2+I_{13}^2+2I_{23}^2)Y^2Z^2-\frac{1}{2I_{23}I_{13}}(2I_{13}^2+I_{22}^2+I_{23}^2)X^2Z^2\\
&-\frac{1}{2I_{13}I_{23}}(I_{11}^2-2I_{11}I_{22}+I_{13}^2+I_{22}^2+I_{23}^2)X^2Y^2.
\end{split}
\]
and
\[
\begin{split}
p_4^{(3)}(X,Y,Z)&= I_{13}X^3-I_{23}Y^3+(4I_{11}-I_{22})Y^2Z-\frac{1}{I_{23}}(I_{11}^2+I_{11}I_{22}+I_{13}^2+I_{23}^2)YZ^2\\
& -\frac{1}{I_{23}}(I_{11}^2-3I_{11}I_{22}+I_{13}^2+2I_{22}^2-2I_{23}^2)X^2Y+(I_{11}-4I_{22})X^2Z \\
& +(I_{11}I_{22}+I_{13}^2+I_{22}^2+I_{23}^2)XZ^2+(I_{11}-I_{22})Z^3 \\ & -\frac{1}{I_{13}I_{23}}\left( I_{11}^2I_{22}
 -I_{11}^2I_{22}+3I_{11}I_{23}^2-3I_{13}^2I_{22}\right)XYZ\\ &+\frac{1}{I_{13}}\left(2I_{11}^2-3I_{11}I_{22}-2I_{13}^2+I_{22}^2+I_{23}^2\right)XY^2,
\end{split}
\]

\end{small}

\begin{proposition}\label{Propoq7}
The momentum locus $\mathfrak{u}_{\varepsilon}^{(\infty)}\subset \so(3)^*$ defined by  $\ell^{(\infty,\varepsilon)}_d$ is contained in
the zero locus of the degree seven polynomial $q_7(M_1,M_2,M_3; \varepsilon)$ given by
\begin{equation*}
\begin{split}
q_7(M_1,M_2,M_3;\varepsilon)=&-4 (I_{13}I_{22}M_1+I_{11}I_{23}M_2-I_{11}I_{22}M_3)q_7^{(4)}(M_1,M_2,M_3) \\
& \qquad +\varepsilon q_7^{(6)}(M_1,M_2,M_3)+\varepsilon^2 q_7^{(7)}(M_1,M_2,M_3),
\end{split}
\end{equation*}
where  for $j=4,6,7$,  $q_7^{(j)}$ denotes a homogeneous polynomial of degree $j$. 
\end{proposition}
Similar to  Proposition \ref{4orderp}, this result follows by a direct substitution of \eqref{MParameInfty} into $q_7$, and was obtained with the
help of a symbolic mathematical software. The polynomials $q_7^{(j)}$ are

\begin{small}
\[
\begin{split}
q_7^{(7)}(X,Y,Z)&=Z^3\left(I_{11}^2I_{22}X^2+2I_{11}I_{13}I_{22}XZ+I_{11}I_{22}^2Y^2+2I_{11}I_{22}I_{23}YZ \right .\\
& \left . +(I_{11}I_{23}^2-I_{13}^2I_{22})Z^2\right)^2.
\end{split}
\]

\[
\begin{split}
&q_7^{(4)}(X,Y,Z)= I_{13}^2I_{22}^4X^4+2I_{11}I_{13}I_{22}^3I_{23}X^3Y  +(2I_{11}^2I_{13}I_{22}^3-2I_{11}I_{13}I_{22}^4+2I_{13}^3I_{22}^2)X^3Z\\ 
 &+ \left(I_{11}^4I_{22}^2-2I_{11}^3I_{22}^3+I_{11}^2I_{13}^2 I_{22}^2+I_{11}^2 I_{22}^4 +I_{11}^2 I_{22}^2I_{23}^2\right)X^2Y^2+6I_{11} I_{13}^2I_{22}^2I_{23}X^2YZ  \\ 
 &  +\left(I_{11}^4I_{22}^2-2I_{11}^3I_{22}^3+2I_{11}^2I_{13}^2I_{22}^2+I_{11}^2I_{22}^4 +I_{11}^2I_{22}^2I_{23}^2-2I_{11}I_{13}^2I_{22}^3+I_{13}^4I_{22}^2  \right .  \\
&   \left . +I_{13}^4I_{22}^2  +I_{13}^2I_{22}^4+I_{13}^2I_{22}^2I_{23}^2\right)X^2Z^2      +2I_{11}^3I_{13}I_{22}I_{23}XY^3 +6I_{11}^2I_{13}I_{22}I_{23}^2XY^2Z       \\
&  +\left( 2I_{11}^3I_{13}I_{22}I_{23}-4I_{11}^2I_{13}I_{22}^2I_{23}+2I_{11}I_{13}^3I_{22}I_{23}+2I_{11}I_{13}I_{22}^3I_{23}+2I_{11}I_{13}I_{22}I_{23}^3\right)XYZ^2 \\ 
&+\left(2I_{11}^2I_{13}I_{22}^3+42I_{11}^2I_{13}I_{22}I_{23}^3 -2I_{11}I_{13}I_{22}^4-2I_{11}I_{13}I_{22}^2I_{23}^2+I_{13}^3I_{22}^3\right)XZ^3 \\
&+I_{11}^4I_{23}^2Y^4+(-2I_{11}^4I_{22}I_{23}+2I_{11}^3I_{22}^2I_{23}+2I_{11}^3I_{23}^3)Y^3Z +(I_{11}^4I_{22}^2+I_{11}^4I_{23}^2-2I_{11}^3I_{22}^3\\
&-2I_{11}^3I_{22}I_{23}^2+I_{11}^2I_{13}^2I_{22}^2+ I_{11}^2I_{13}^2I_{23}^2+I_{11}^2I_{22}^4+2I_{11}^2I_{22}^2I_{23}^2+I_{11}^2I_{23}^2)Y^2Z^2 \\
& +(2I_{11}^3I_{22}^2I_{23}-2I_{11}^4I_{22}I_{23}+2I_{11}^3I_{23}^3 -2I_{11}^2I_{13}^2I_{22}I_{23}+4I_{11}I_{13}^2I_{22}^2I_{23})YZ^3+(I_{11}^4I_{22}^2 \\
& -2I_{11}^3I_{22}^3-2I_{11}^3I_{22}I_{23}^2 +2I_{11}^2I_{13}^2I_{22}^2 +   I_{11}^2I_{22}^4+2I_{11}^2I_{22}^2I_{23}^2+I_{11}^2I_{23}^4-2I_{11}I_{13}^2I_{22}^3+ \\
& +2I_{11}I_{13}^2I_{22}I_{23}^2+I_{13}^4I_{22}^2)Z^4.
 \end{split}
\]

\[
\begin{split}
&q_7^{(6)}(X,Y,Z)= \left(6(I_{11}^3I_{13}I_{22}^3-I_{11}^2I_{13}I_{22}^4)X^4YZ-6I_{11}^2I_{13}I_{22}^3I_{23}X^4Z^2 \right.\\ &+2(I_{11}^5I_{22}^2-3I_{11}^4I_{22}^3
 +3I_{11}^3I_{22}^4 -I_{11}^2I_{22}^5)X^3Y^3+6(I_{11}^3I_{22}^3I_{23}-I_{11}^2I_{22}^4I_{23})X^3Y^2Z \\
 &+ 2(I_{11}^5I_{22}^2-4I_{11}^4I_{22}^3  +I_{11}^3I_{13}^2I_{22}^2+I_{11}^3I_{22}^4+4I_{11}^2I_{13}^2I_{22}^3-3I_{11}^2I_{22}^3I_{23}^2\\
 &-2I_{11}I_{13}^2I_{22}^4)X^3YZ^2-2(I_{11}^4I_{22}^2I_{23}-3I_{11}^3I_{22}^3I_{23} +5I_{11}^2I_{13}^2I_{22}^2I_{23}+I_{11}^2I_{22}^2I_{23}^3 \\
& +2I_{11}I_{13}^2I_{22}^3I_{23})X^3Z^3 +6(I_{11}^4I_{13}I_{22}^2-I_{11}^3I_{13}I_{22}^3)X^2Y^3Z\\
& +6(I_{11}^3I_{13}I_{22}^2I_{23}-I_{11}^2I_{13}I_{22}^3I_{23})X^2Y^2Z^2+2(I_{13}^3I_{22}^4+3I_{11}^4I_{13}I_{22}^2-3I_{11}^3I_{13}I_{22}^3\\
& -3I_{11}^3I_{13}I_{22}I_{23}^2+3I_{11}^2I_{13}^3I_{22}^2-2I_{11}^2I_{13}I_{22}^4-8I_{11}^2I_{13}I_{22}^2I_{23}^2+3I_{11}I_{13} 3I_{22}^3)X^2YZ^3\\
&+ 2(I_{13}^3I_{22}^3I_{23}+6I_{11}^3I_{13}I_{22}^2I_{23}-2I_{11}^2I_{13}I_{22}^3I_{23}-5I_{11}^2I_{13}I_{22}I_{23}^3-6I_{11}I_{13}^3I_{22}^2I_{23})X^2Z^4\\
&  +6(I_{11}^4I_{22}^2I_{23}-I_{11}^3I_{22}^3I_{23})XY^4Z-2(I_{11}^2I_{22}^5+I_{11}^2I_{22}^2I_{23}^2-2I_{11}^4I_{22}I_{23}^2+3I_{11}^4I_{22}^3\\
&  -3I_{11}^3I_{13}^2I_{22}^2-4I_{11}^3I_{22}^4+4I_{11}^3I_{22}^2I_{23}^2)XY^3Z^2 -2(I_{11}^4I_{23}^3-2I_{11}^4I_{22}^2I_{23}-3I_{11}^3I_{22}^3I_{23}\\
&  +3I_{11}^3I_{22}I_{23}^3-8I_{11}^2I_{13}^2I_{22}^2I_{23}+3I_{11}^2I_{22}^4I_{23}+3I_{11}^2I_{22}^2I_{23}^3-3I_{11}I_{13}^2I_{22}^3I_{23})XY^2Z^3\\
& +2(I_{11}I_{13}^2I_{22}^2I_{23} 2-4I_{11}^4I_{22}^3+5I_{11}^4I_{22}I_{23}^2+3I_{11}^3I_{13}^2I_{23}+4I_{11}^3I_{22}^4+2I_{11}^3I_{22}^2I_{23}^2\\
& -2I_{11}^3I_{22}^4-2I_{11}^2I_{13}^2I_{23}^3-I_{11}^2I_{13}^2I_{22}I_{23}^2-3I_{11}^2I_{22}^3I_{23}^2-3I_{11}^2I_{22}I_{23}^4+3I_{11}I_{13}^4I_{22}^2\\
&  -5I_{11}I_{13}^2I_{22}^4+2I_{13}^4I_{22}^3)XYZ^4-2(I_{11}^2I_{22}^2I_{23}^3+I_{11}^2I_{23}^5+I_{13}^4I_{22}^2I_{23}+2I_{11}I_{13} 2I_{22}I_{23}^3\\
& +5I_{11}I_{13}^2I_{22}^3I_{23}-7I_{11}^2I_{13}^2I_{22}I_{23}+4I_{11}^4I_{22}^2I_{23}-4I_{11}^3I_{22}^3I_{23}-3I_{11}^3I_{22}I_{23}^3)XZ^5\\
& +6I_{11}^3I_{13}I_{22}^2I_{23}Y^4Z^2+2(I_{11}^2I_{13^3}I_{23}^2+I_{11}^2I_{13}I_{22}^4-3I_{11}^3I_{13}I_{22}^3+2I_{11}^3I_{13}I_{22}I_{23} 2\\
&+5I_{11}^2I_{13}I_{22}^2I_{23}^2)Y^3Z^3-2(I_{11}^3I_{13}I_{23}^3+6I_{11}^2I_{13}I_{22}^3I_{23} -2I_{11}^3I_{13}I_{22}^2I_{23}-6I_{11}^2I_{13}I_{22}I_{23}^3\\
&-5I_{11}I_{13}^3I_{22}^2I_{23})Y^2Z^4+2(I_{11}^2I_{13}^3I_{22}^2+I_{11}^2I_{13}I_{23}^4+I_{13}^5I_{22}^2+2I_{11}I_{13}^3I_{22}I_{23}^2\\
&-3I_{11}I_{13}^3I_{22}^3+4I_{11}^2I_{13}I_{22}^4-7I_{11}^2I_{13}I_{22}^2I_{23}^2-4I_{11}^3I_{13}I_{22}^3+5I_{11}^3I_{13}I_{22}I_{23}^2)YZ^5\\
& \left . 8I_{11}I_{13}I_{22}I_{23}(I_{11}I_{22}^2-I_{11}^2I_{22})Z^6\right)
\end{split}
\]

\end{small}


\begin{thebibliography}{99}





\bibitem{Arn} 
     \newblock V. I. Arnold, V. V. Kozlov and A. I. Neishtadt,
     \newblock \emph{Mathematical Aspects of Classical and Celestial Mechanics; Dynamical Systems III},
     \newblock 3$^{rd}$ edition, Springer-Verlag, New York, (1994).

\bibitem{Bl}  
     \newblock A. M. Bloch,
     \newblock \emph{Nonholonomic Mechanics and Control},
     \newblock 2$^{nd}$ edition, Springer-Verlag, New York, (2003).



\bibitem{BoSu}
    \newblock  A. I. Bobenko  and Y. B. Suris, 
    \newblock Discrete Lagrangian reduction, discrete Euler-Poincar\'e equations and semidirect products,
    \newblock \emph{Lett. Math. Phys.}, \textbf{49} (1999), 79--83.





\bibitem{CoMa} 
    \newblock  J. Cort\'es and S. Mart\'inez,
    \newblock Nonholonomic integrators,
    \newblock \emph{Nonlinearity}, \textbf{14} (2001),  1365--1392.



\bibitem{FeKoz} 
    \newblock   Y. N. Fedorov  and V. V. Kozlov,
    \newblock Various aspects of $n-$dimensional rigid body dynamics,
    \newblock  \emph{Amer. Math. Soc. Transl.}, \textbf{168} (1995),  141--171.







\bibitem{FeZen} 
    \newblock  Y. N. Fedorov  and  D. V. Zenkov,  
    \newblock Discrete nonholonomic {LL} systems on {L}ie groups,
    \newblock \emph{Nonlinearity}, \textbf{18} (2005),   2211--2241.




\bibitem{FeChap}
    \newblock   Y. N. Fedorov,
        \newblock A discretization of the nonholonomic Chaplygin sphere problem,
    \newblock \emph{SIGMA}, \textbf{3} (2007),   15pp.


\bibitem{FeMaPr} 
    \newblock   Y. N. Fedorov, A. J. Maciejewski  and M. Przybylska,
        \newblock The Poisson equations in the nonholonomic Suslov problem: integrability, meromorphic
and hypergeometric solutions,
    \newblock \emph{Nonlinearity}, \textbf{22} (2009),   2231--2259.

\bibitem{GNetal}  
    \newblock   L. C. Garc\'ia-Naranjo, J. C. Marrero, A. J. Maciejewski  and M. Przybylska,
        \newblock The inhomogeneous Suslov problem,
    \newblock \emph{Phys. Lett. A}, \textbf{378} (2014),   2389--2394.


\bibitem{Iglesias} 
\newblock D. Iglesias, J. C. Marrero, D. Mart\'in de Diego and E. Mart\'inez,
 \newblock Discrete nonholonomic Lagrangian systems on Lie groupoids,
 \emph{J. Nonlinear Sci.}, \textbf{18} (2008),   351-397.
 


\bibitem{Iserles} 
    \newblock A. Iserles, H. Z. Munthe-Kaas, S. P. Norsett  and A. Zanna,
        \newblock Lie-group methods,
    \newblock \emph{Acta Numerica}, \textbf{9} (2000),   215--365.


\bibitem{JiSch} 
    \newblock F. Jim\'enez F and J. Scheurle,
        \newblock On the discretization of nonholonomic mechanics in $\R^n$,
    \newblock \emph{J.  Geom. Mech.}, \textbf{7} (2015),   43--80.



\bibitem{JiSch2}
\newblock F. Jim\'enez F and J. Scheurle,
 \newblock On the discretization of the Euler-Poincar\'e-Suslov equations in $SO(3)$,
    \newblock arXiv  506.01289. To appear in \emph{J.  Geom. Mech.}




\bibitem{MaRa}  
\newblock J. E. Marsden and T. S. Ratiu,
 \newblock \emph{Introduction to Mechanics and Symmetry. A Basic Exposition of Classical Mechanical Systems},
    \newblock   2$^{nd}$ edition, Texts in Applied Mathematics, \textbf{17}. Springer-Verlag, New York, 1999.





\bibitem{MarsdenWest} 
    \newblock J. E. Marsden  and  M. West,
        \newblock Discrete mechanics and variational integrators,
    \newblock \emph{Acta Numerica}, \textbf{10} (2001),    357--514.


\bibitem{McPer} 
    \newblock R. McLachlan  and M. Perlmutter,
        \newblock Integrators for nonholonomic mechanical systems,
    \newblock \emph{J. Nonlinear Sci.}, \textbf{16} (2006),    283--328.

\bibitem{Mose} 
    \newblock J. Moser  and A. P. Veselov,
        \newblock Discrete versions of some classical
integrable systems and factorization of matrix polynomials,
    \newblock \emph{Comm. Math. Phys.}, \textbf{139} (1991),    217--243.

\bibitem{Suslov}
\newblock G. K. Suslov, \emph{Theoretical Mechanics}, Gostekhizdat, Moscow, 1946 (in Russian).


\bibitem{Vese} 
    \newblock  A. P. Veselov,
        \newblock Integrable discrete-time systems and difference operators,
    \newblock \emph{Funct. Anal. Appl.}, \textbf{22} (1988),   1--13.



\end{thebibliography}
\end{document}